%% file: main.tex
\documentclass[preprint,twoside,11pt]{article}

\usepackage{amsthm}
\usepackage{jmlr2e}

\def\ContraLinguaLatina{} 
\input{preamble}

\usepackage{lastpage}
\jmlrheading{23}{2022}{1-\pageref{LastPage}}{08/01}{Under Review}{roddenberry22}{T. Mitchell Roddenberry and Santiago Segarra}
\ShortHeadings{Limits of Dense Simplicial Complexes}{Roddenberry and Segarra}
\firstpageno{1}

\hypersetup{
  pdfauthor={T. M. Roddenberry \& S. Segarra},
  pdftitle={Limits of Dense Simplicial Complexes}
}

\begin{document}

\title{Limits of Dense Simplicial Complexes}

\author{\name T.~Mitchell~Roddenberry \email mitch@rice.edu \\
  \addr Department of Electrical and Computer Engineering \\
  Rice University \\
  Houston, TX 77005-1827, USA
  \AND
  \name Santiago~Segarra \email segarra@rice.edu \\
  \addr Department of Electrical and Computer Engineering \\
  Rice University \\
  Houston, TX 77005-1827, USA}

\editor{Editor TBD}

\maketitle

\begin{abstract}%
  We develop a theory of limits for sequences of dense abstract simplicial complexes, where a sequence is considered convergent if its homomorphism densities converge.
  The limiting objects are represented by stacks of measurable $[0,1]$-valued functions on unit cubes of increasing dimension, each corresponding to a dimension of the abstract simplicial complex.
  We show that convergence in homomorphism density implies convergence in a cut-metric, and vice versa, as well as showing that simplicial complexes sampled from the limit objects closely resemble its structure.
  Applying this framework, we also partially characterize the convergence of nonuniform hypergraphs.
\end{abstract}



\input{sections/10_intro}
\input{sections/20_prelim}
\input{sections/30_complexons}

\section{Main Results}
\label{sec:space}

In this section, we characterize the space of complexons equipped with the cut-distance, and show how the cut-distance relates to the homomorphism densities.
In particular, we show that they induce equivalent, compact topologies on the space of complexons.
Our treatment closely follows that of~\citet{Lovasz2006,Borgs2008}.

\input{sections/41_counting-lemma}
\input{sections/42_sampling-lemma}
\input{sections/43_inverse-counting-lemma}
\input{sections/44_compactness}

\section{Discussion and Remarks}
\label{sec:remarks}

By \cref{thm:hom-cut-eq}, homomorphism densities are sufficient to characterize complexons in the cut-metric.
This yields a characterization of limit objects for sequences of simplicial complexes.
Indeed, recall that sequence $(K_n)$ for $n\geq 1$ of simplicial complexes is said to be convergent if for all simplicial complexes $F$, the sequence $(\thom(F,K_n))_{n\geq 1}$ is convergent.
Equivalently, by \cref{lemma:finite-to-complexon-cut}, the sequence $(K_n)$ is convergent if for all simplicial complexes $F$, the sequence $(\thom(F,W_{K_n}))$ is convergent.
Applying \cref{thm:hom-cut-eq}, this is equivalent to convergence of the sequence $(W_{K_n})$ in the canonical topology.
Since $\kerns$ is compact in the canonical topology (\cf{}~\cref{thm:compactness}), the limit of the sequence $(W_{K_n})$ is itself an element $W\in\kerns$.
It is in this sense that we say $W$ (or any equivalent complexon in the canonical topology) is ``the limit of'' the sequence of simplicial complexes $(K_n)$ as $n\to\infty$.

In the remainder of the paper, we discuss how complexons relate to some other notions in the literature.

\input{sections/51_posets}
\input{sections/52_local-random}
\input{sections/52_5_other-random}
\input{sections/53_hypergraphs}

\acks{We would like to acknowledge support for this project from the National Science Foundation (NSF grant CCF-2008555).}

\appendix

\input{sections/60_proofs.tex}

\vskip 0.2in
\bibliography{bibliography}

\end{document}

%% file: preamble.tex
\usepackage{mathtools,accents}
\usepackage{algorithm,algorithmic}

\usepackage{float}
\usepackage{wrapfig}
\usepackage[inline]{enumitem}

\usepackage{booktabs}

\usepackage[capitalise,nameinlink]{cleveref}
\usepackage{doi}

\Crefformat{figure}{#2Figure~#1#3}

\usepackage[x11names]{xcolor}

\newcounter{todo}

\makeatletter

\newcommand\listoftodos{\section*{To-do List}\@starttoc{tod}}
\makeatother

\usepackage{tikz}
\usetikzlibrary{cd}
\usetikzlibrary{shapes.geometric}
\usetikzlibrary{arrows.meta}
\usepgflibrary{plotmarks}

\usetikzlibrary{positioning}
\usetikzlibrary{backgrounds}
\usetikzlibrary{matrix}
\usetikzlibrary{decorations.markings}
\usetikzlibrary{calc}

\usepackage{pgfplots}
\pgfplotsset{compat=1.7}
\usepgfplotslibrary{groupplots}
\usepgfplotslibrary{fillbetween}
\usetikzlibrary{plotmarks}

\newtheorem{example}{Example}
\newtheorem{counterexample}[example]{Counterexample}

\newtheorem{theorem}{Theorem}
\newtheorem{lemma}[theorem]{Lemma}
\newtheorem{proposition}[theorem]{Proposition}
\newtheorem{corollary}[theorem]{Corollary}

\newtheorem*{example*}{Example}
\newtheorem*{theorem*}{Theorem}

\theoremstyle{definition}
\newtheorem{definition}[theorem]{Definition}

\theoremstyle{remark}

\DeclareMathOperator{\supp}{supp}

\DeclareMathOperator{\Tub}{Tub}

\newcommand{\hintCO}[1]{\ensuremath{[#1)}}

\newcommand{\djunion}{\ensuremath{\bigsqcup}} 
\newcommand{\drv}{\ensuremath{\mathrm{d}}} 
\newcommand{\reals}{\ensuremath{\mathbb{R}}} 
\newcommand{\indic}{\ensuremath{\mathbf{1}}} 

\newcommand{\labelcutd}{\ensuremath{\mathrm{d}}}
\newcommand{\unlabelcutd}{\ensuremath{\widehat{\delta}}}
\newcommand{\cutd}{\ensuremath{\delta}}
\newcommand{\weakcutd}{\ensuremath{\gamma}}

\newcommand{\dvar}{\ensuremath{{\mathrm{d}}_{\mathrm{var}}}}

\newcommand{\faceted}[1]{\ensuremath{#1^{\circ}}} 
\newcommand{\antifaceted}[1]{\ensuremath{#1_{\circ}}} 

\newcommand{\sample}[1]{\ensuremath{\mathbb{#1}}} 
\newcommand{\partition}[1]{\ensuremath{\mathcal{#1}}} 
\newcommand{\vect}[1]{\ensuremath{\mathbf{#1}}} 

\newcommand{\cut}{\ensuremath{\square}}
\newcommand{\ncut}[1]{\ensuremath{\square,{#1}}}

\newcommand{\hlower}[1]{\ensuremath{\lfloor #1\rfloor}}
\newcommand{\hupper}[1]{\ensuremath{\lceil #1\rceil}}

\newcommand{\kerns}{\ensuremath{\mathcal{W}}}
\newcommand{\fkerns}{\ensuremath{\mathcal{W}_0}}

\newcommand{\Borel}{\ensuremath{\mathcal{B}}}
\renewcommand{\Pr}{\ensuremath{\mathbb{P}}}
\newcommand{\Ex}{\ensuremath{\mathbb{E}}}

\renewcommand{\hom}{\ensuremath{\mathsf{hom}}}

\newcommand{\ind}{\ensuremath{\mathsf{ind}}}

\newcommand{\thom}{\ensuremath{t}}

\newcommand{\tind}{\ensuremath{t_{\ind}}}

\renewcommand{\vert}{\ensuremath{\mathsf{v}}}
\renewcommand{\Vert}{\ensuremath{\mathsf{V}}}

\newcommand{\set}[1]{\ensuremath{\mathcal{\uppercase{#1}}}}

\newcommand{\ie}{\textit{i.e.}}
\newcommand{\eg}{\textit{e.g.}}
\newcommand{\cf}{\textit{cf.}}

\ifdefined\ContraLinguaLatina
\renewcommand{\ie}{that is}
\renewcommand{\eg}{for example}
\renewcommand{\cf}{see}
\fi

\newcommand{\iid}{\textnormal{i.i.d.}}

%% file: sections/10_intro.tex
\section{Introduction}
\label{sec:intro}

The theory of graph limits has been one of the most fruitful developments in modern combinatorics, with applications ranging from extremal graph theory to statistical physics~\citep{Borgs2008,Borgs2012,Lovasz2012}.
Viewing graphs as (abstract) simplicial complexes of dimension one, it is natural to ask how this theory can be extended to limiting objects of simplicial complexes.
Indeed, recent work in stochastic topology has studied the topological properties of large random simplicial complexes, characterizing their connectivity and (co)homology~\citep{Bobrowski2022}.

An example of such a random simplicial complex is a random geometric \v{C}ech complex.
\input{sections/10_5_cech-example}

The characterization of convergence via the distribution of induced substructures is closely related to the theory of graphons~\citep{Borgs2008}.
Following this approach, we develop an analogous limit theory for simplicial complexes.
We consider a sequence of simplicial complexes to be convergent if their homomorphism densities converge, which we show to be equivalent to convergence in a cut-metric.
The limiting object of such a convergent sequence is described in terms of a collection of symmetric kernel functions on the sequence of unit cubes of increasing dimension; we dub these objects \emph{complexons}.\footnote{Following the concluding remarks of \citet{Bobrowski2022}.}
These objects closely relate to graphons in the random complex models that they induce, with two representations that correspond to either the usual presentation of a simplicial complex as a collection of finite sets closed under restriction, or the presentation in terms of its facets.
The main results of this work relate the cut-metric and homomorphism densities for the limit objects of large, dense simplicial complexes.
\begin{theorem*}[Informal]
  If two complexons are close in the cut-metric, then they yield similar random simplicial complexes by sampling.
  Conversely, if they yield similar random simplicial complexes by sampling, then they are close in the cut-metric.
\end{theorem*}
This is stated more formally in \cref{thm:hom-cut-eq}.
As part of the proof, it will also be shown that large simplicial complexes drawn from a complexon concentrate in the cut-metric, and thus have similar homomorphism densities to the complexon.

The paper is organized as follows.
After covering preliminaries in \cref{sec:prelim}, complexons and random sampling from complexons are defined in \cref{sec:complexons}.
The main results are stated in \cref{sec:space}, with most proofs relegated to the appendix.
Some remarks relating complexons to other notions in the literature, including the convergence of nonuniform hypergraphs under certain conditions, are made in \cref{sec:remarks}.

\subsection{Related Work}

In graph limit theory, a key idea is to reduce extremely large graphs into simpler objects by treating them as distributions of random subgraphs.
The study of such limits was initiated by~\citet{Lovasz2006}, and further developed by~\citet{Borgs2008,Borgs2012}.
These notions relate to exchangeable random arrays, considered originally by~\citet{Aldous1981,Hoover1979} and related to random graphs by~\citet{Diaconis2007}.
The work of~\citet{Elek2012} extended this approach to dense, uniform hypergraphs by considering analogous notions of homomorphism convergence and cut-metrics, which was further studied by~\citet{Zhao2015}.
The general approach of relating homomorphism densities and some cut-metric has been applied to form limit theories for other combinatorial objects as well, such as partially ordered sets~\citep{Janson2011} and random cographs~\citep{Stufler2021}.

The theory of large, random simplicial complexes for modeling higher-order interactions in network science has been quite active recently, with approaches from topological data analysis, dynamical systems, and signal processing~\citep{Schaub2021,Battiston2022}.
A series of recent papers considered the topological properties of large, random simplicial complexes~\citep{Costa2016,Costa2017a,Costa2017b}.
Indeed, a type of limiting object for simplicial complexes similar to the Rado graph was considered by~\citet{Farber2021}, which arises as the limit of a dense random simplicial complex on countably many nodes, in a certain regime.
In the context of TDA, the study of persistent homology of random subsamples of metric measure spaces has been considered in great detail by~\citet{Chazal2014,Chazal2015,Chazal2017}.
In particular, the persistence homology of persistence barcodes of random subsets of compact metric measure spaces is considered.
This is in contrast to our work in two ways.
First, the simplicial complexes we consider are not required to be modeled by an underlying metric measure space.
Second, complexons naturally model simplicial complexes where not only the nodes are random (in that the nodes themselves are positioned in space according to some distribution), but where the faces connecting them are random.
Initial studies of random complexes in noncompact spaces have been undertaken by \citet{Adler2014,Owada2020}, where the (persistent) homological structure of point clouds drawn from distributions of unbounded support in $\mathbb{R}^d$ are considered.
In the observation of what they dub ``topological crackle,'' the robustness of TDA methods to potentially unbounded noise is studied.
Both of these works consider simplicial complexes determined via fixed rules applied to random sets of points, which varies from our study of complexes that are random even when conditioned on a fixed set of points.
We refer the reader to the excellent survey paper of~\citet{Bobrowski2022} for more references on the emerging topic of large, random simplicial complexes.

%% file: sections/10_5_cech-example.tex
Let $\mu$ be a probability measure on $\mathbb{R}^d$ for some integer $d\geq 1$, and let $\epsilon>0$ be a real number.
Denote the support of $\mu$ by $\mathcal{X}=\supp(\mu)$, and treat $\mathcal{X}$ as a metric probability space, with the metric inherited from $\mathbb{R}^d$.
For some fixed $\epsilon>0$, we consider the \v{C}ech complex of $\mathcal{X}$ in $\mathbb{R}^d$ with radius $\epsilon$, which is defined as the infinite simplicial complex $\check{C}_\epsilon(\mathcal{X},\mathbb{R}^d)$ composed of all finite subsets of $\mathcal{X}$ with diameter strictly less than $\epsilon$.

\begin{wrapfigure}{L}{0.4\textwidth}
  \centering
  \includegraphics[width=\linewidth]{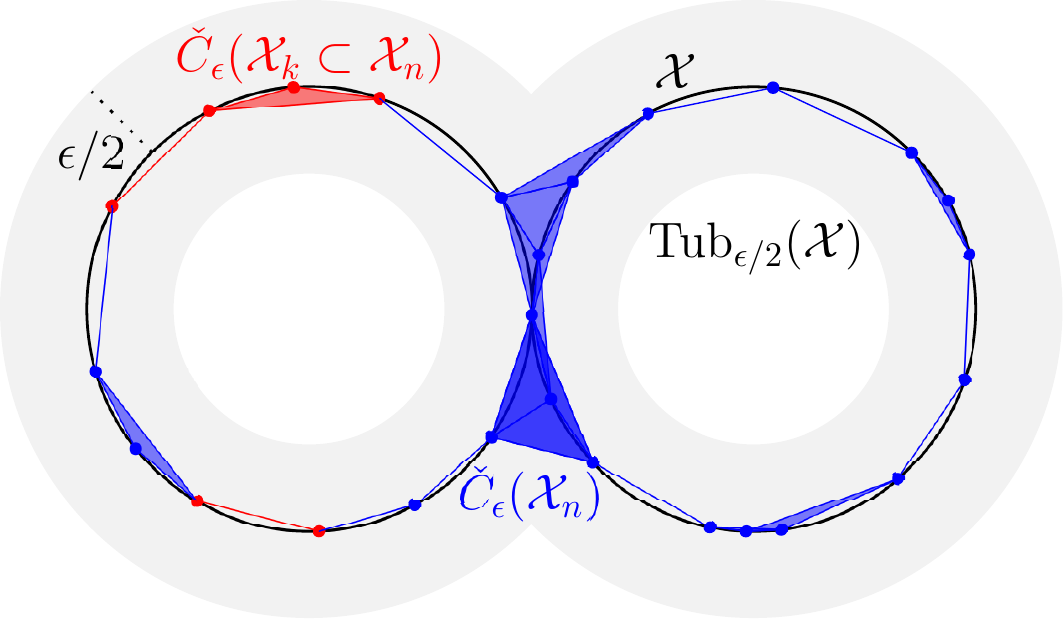}
  \caption{\v{C}ech complexes formed from a bouquet of two circles in $\mathbb{R}^2$.}
  \label{fig:bouquet-example}
\end{wrapfigure}

Denoting by $\Tub_{\epsilon/2}(\mathcal{X},\mathbb{R}^d)$ the set of all points in $\mathbb{R}^d$ that are less than $\epsilon/2$-away from $\mathcal{X}$, the set of all open balls of radius $\epsilon/2$ centered at points in $\mathcal{X}$ is a cover of $\Tub_{\epsilon/2}(\mathcal{X},\mathbb{R}^d)$ by contractible open sets.
The nerve theorem then implies that the geometric realization of the \v{C}ech complex $\check{C}_\epsilon(\mathcal{X},\mathbb{R}^d)$ is homotopy equivalent to $\Tub_{\epsilon/2}(\mathcal{X},\mathbb{R}^d)$.
As an example, let $\mathcal{X}$ be a bouquet of two unit circles in $\mathbb{R}^2$, \ie{}, the union of two circles of unit radius centered at $(-1,0)$ and $(1,0)$ respectively, as pictured in \cref{fig:bouquet-example}, and let $\mu$ be the uniform measure on $\mathcal{X}$.
Clearly, the space $\Tub_{\epsilon/2}(\mathcal{X},\mathbb{R}^2)$ (pictured in \cref{fig:bouquet-example} as the gray shaded region) is homotopy equivalent to $\mathcal{X}$ for any $\epsilon<2$.

The field of topological data analysis (TDA)~\citep{Carlsson2009} has leveraged this idea to great success, replacing troublesome computations about the topological structure of complicated spaces with more tractable computations on approximations formed from point clouds, such as the \v{C}ech, or more commonly, Vietoris-Rips complex.
In practice, TDA applications typically do not have access to the entire space $\mathcal{X}$; rather, a finite set of $n$ points $\mathcal{X}_n=\{X_1,\ldots,X_n\}$ in $\mathcal{X}$ sampled \iid{} according to the probability measure $\mu$ is available, from which the \v{C}ech complex $\check{C}_\epsilon(\mathcal{X}_n,\mathbb{R}^d)$ is formed.
This \v{C}ech complex is used as an ``estimate'' of sorts for $\check{C}_\epsilon(\mathcal{X},\mathbb{R}^d)$.
If the space $\mathcal{X}$ is compact, one can show that as $n\to\infty$, the space $\mathcal{X}_n$ with the Euclidean metric almost surely converges to $\mathcal{X}$ in the Gromov-Hausdorff metric.
This implies, for instance, convergence of the persistence diagrams of $\mathcal{X}_n$ to that of $\mathcal{X}$ in the bottleneck distance~\citep{Chazal2009}.
Returning to our example of the bouquet of circles, see in \cref{fig:bouquet-example} that a finite subset $\mathcal{X}_n\subset\mathcal{X}$ yields a \v{C}ech complex that preserves the topological features of the space, \eg{}, having a free fundamental group of rank two.
Observe that sampling more points would not introduce any additional topological features to the constructed complex.

However, convergence of this sort is not as easy to establish when $\mathcal{X}$ is not compact.
For instance, the case when $\mathbb{X}=\mathbb{R}^d$ and $\mu$ is a radially symmetric power-law, exponential, or Gaussian distribution was considered by~\citet{Adler2014}, where they showed that under many circumstances, highly nontrivial homological features emerge in ways that vary with $n$, precluding the convergence results mentioned above.
We consider an alternative way to describe the convergence of the \v{C}ech complexes $\check{C}_\epsilon(\mathcal{X}_n,\mathbb{R}^d)$ to $\check{C}_\epsilon(\mathcal{X},\mathbb{R}^d)$.
Observe that each $\check{C}_\epsilon(\mathcal{X}_n,\mathbb{R}^d)$ is a random induced subcomplex of $\check{C}_\epsilon(\mathcal{X},\mathbb{R}^d)$, with nodes distributed \iid{} according to the probability measure $\mu$.
This motivates a preliminary definition of convergence for simplicial complexes: we say that a sequence of finite simplicial complexes $K_n$ is convergent if for all $k\geq 1$ and all finite simplicial complexes $F$ with vertex set $[k]$, the probability that the induced subcomplex of $k$ nodes chosen uniformly at random from $K_n$ is equal to $F$ is convergent.
That is to say, the distribution of uniformly sampled induced subcomplexes of any fixed size converges as $n\to\infty$.
We indicate this further subsampling operation in red in \cref{fig:bouquet-example}.
In this case, for instance, it holds for all $n$ that the probability of subsampling a cycle on $k<\lfloor\pi/\arcsin(\epsilon/2)\rfloor$ nodes from $\check{C}_\epsilon(\mathcal{X}_n,\mathbb{R}^2)$ is zero.

%% file: sections/20_prelim.tex
\section{Preliminaries}
\label{sec:prelim}

The notation $[n]$ denotes the set of integers $\{1,2,\ldots,n\}$.
A simplicial complex\footnote{Strictly speaking, we are concerned with abstract simplicial complexes. However, we will never refer to the geometric realization, so we will omit the word \emph{abstract}.}
$K$ is a finite collection of finite sets that is closed under restriction, \ie, the taking of subsets.
The collection of sets in $K$ with cardinality $d+1$ is denoted by $K^{(d)}$, and we call that set the set of $d$-simplices.
For a $d$-simplex $\sigma$, we say that $\dim{\sigma}=d$.
The set $K^{(0)}$ is referred to as the set of nodes, with the alternative notation $\Vert(K)$, when convenient.
We denote the cardinality of $\Vert(K)$ by $\vert(K)$.
Although $K^{(0)}$ is a set of singleton sets, we will find it useful sometimes to identify said singleton sets with their constituent elements.
The sets $K^{(1)}$ and $K^{(2)}$ are sometimes referred to as the sets of edges and triangles, respectively.
If $K$ is such that for all $c>d$, $K^{(c)}=\emptyset$, then we say that $K$ is a simplicial complex of dimension $d$.
A simple graph, for instance, is a simplicial complex of dimension one.
It will also be useful to consider weighted simplicial complexes.
A weighted simplicial complex $(H,\omega)$ is taken to be the power set $H=2^{\Vert(H)}$ for some finite $\Vert(H)$, and a function $\omega:H\to[0,1]$, where $\omega$ is called the weight function.
It is assumed that $\omega(\Vert(H))=\omega(\emptyset)=1$.

For a simplicial complex $K$, the set of facets of $K$ is the collection of maximal sets in $K$, \ie, the sets in $K$ that are subsets of no other sets in $K$, denoted $\faceted{K}$.
Observe that the set of facets of a simplicial complex completely characterizes the simplicial complex, by taking the closure of the facets under restriction.
The set of antifacets of $K$ is the collection of minimal sets in $2^{\Vert(K)}\setminus K$, \ie, the collections of nodes that do not constitute a simplex, but whose strict subsets do constitute a simplex, denoted $\antifaceted{K}$.
Similarly, the set of antifacets of a simplicial complex completely characterize the simplicial complex.
We  define the facets of a weighted simplicial complex not by changing the underlying set, but by changing the weight function.
For a weighted simplicial complex $(H,\omega)$, the faceted weighted simplicial complex $(\faceted{H},\faceted{\omega})$ is such that $\faceted{H}=H$, and for each $\sigma\in\faceted{H}$,
\begin{equation*}
  \faceted{\omega}(\sigma) = \prod_{\sigma'\subseteq\sigma}\omega(\sigma').
\end{equation*}

\subsection{Homomorphism Densities}

Much like in the theory of dense graph limits, we characterize simplicial complexes in terms of their homomorphism densities.
For two simplicial complexes $F,K$, a homomorphism from $F$ to $K$ is a map $\phi:\Vert(F)\to\Vert(K)$ such that for any $\sigma\in F^{(d)}$, it holds that $\phi(\sigma)\in K^{(d)}$, \ie, a homomorphism is a simplex-preserving map.
An induced homomorphism is an injective homomorphism from $F$ to $K$ with the added condition that $\sigma\in F^{(d)}$ if and only if $\phi(\sigma)\in K^{(d)}$, \ie, an injective homomorphism that also preserves non-simplices.
We denote the number of homomorphisms and induced homomorphisms from $F$ to $K$ by $\hom(F,K)$ and $\ind(F,K)$, respectively.

Normalizing these quantities appropriately yields homomorphism densities.
Homomorphisms densities answer the following type of question: given a uniformly chosen (injective) random map $\phi:\Vert(F)\to \Vert(K)$, what is the probability that that map is an (induced) homomorphism?
More precisely, we respectively define the homomorphism density and induced homomorphism density of $F$ in $K$ as
\begin{align*}
  \thom(F,K) &= \frac{\hom(F,K)}{\vert(F)^{\vert(K)}} \\
  \tind(F,K) &= \frac{\ind(F,K)}{P(\vert(K),\vert(F))},
\end{align*}
where $P(n,k)$ is the number of $k$-permutation of $n$, \ie, the number of injective maps from a set of cardinality $k$ into a set of cardinality $n$.

\subsection{The Cut Distance}

The homomorphism densities characterize the distribution of subcomplexes in a simplicial complex.
This may be thought of as looking at local structures in a simplicial complex.
On the other hand, the cut-metric characterizes simplicial complexes in a global way.
The cut-norm is defined for a multidimensional matrix \citep[Section~6]{Frieze1999} in the following way.
For finite sets $X_1,X_2,\ldots,X_r$, let $A:\prod_j X_j \to\reals$.
For subsets $S_j\subseteq X_j$ for $j\in[r]$, put
\begin{equation*}
  A(S_1,\ldots,S_r) = \sum_{e\in\prod_j S_j}A(e).
\end{equation*}
The (normalized) cut-norm of $A$ is defined as
\begin{equation*}
  \|A\|_\cut = \frac{\max\left\{|A(S_1,\ldots,S_r)|:S_j\subseteq X_j\text{ for }j\in[r]\right\}}{\prod_j|X_j|}.
\end{equation*}

Let $K_1,K_2$ be simplicial complexes such that $\Vert(K_1)=\Vert(K_2)$.
For some $d\geq 1$, let $A_{1,d}:\prod_{j=1}^{d+1}\Vert(K_1)\to\{0,1\}$ be the indicator function of $d$-simplices in $K_1$, \ie, $A_{1,d}(i_1,\ldots,i_{d+1})=1$ if and only if $\{i_1,\ldots,i_{d+1}\}\in K_1^{(d)}$.
Define $A_{2,d}$ in the same way for $K_2$.
The $d$-dimensional labeled cut-distance between $K_1$ and $K_2$ is defined as follows:
\begin{equation*}
  \labelcutd_{\ncut{d}}(K_1,K_2) = \|A_{1,d}-A_{2,d}\|_\cut.
\end{equation*}
We extend this definition to weighted simplicial complexes $(K_3,\omega)$ by putting $A_{3,d}(i_1,\ldots,i_{d+1})=\omega(\{i_1,\ldots,i_{d+1}\})$, and using the normalized cut-norm for the metric as before.
This allows us to compare weighted and unweighted simplicial complexes.

Observe that the $d$-dimensional labeled cut-distance finds the maximal collection of subsets $S_1,\ldots,S_{d+1}\subseteq\Vert(K_1)$ such that the number of $d$-simplices contained in $\prod_j S_j$ differs maximally between $K_1$ and $K_2$.
More precisely, for arbitrary subsets $S_1,\ldots,S_{d+1}\subseteq\Vert(K_1)$, put
\begin{equation*}
  K_1^{(d)}(S_1,\ldots,S_{d+1}) = \left\{\sigma\in K^{(d)}:\sigma\in\prod_j S_j\right\},
\end{equation*}
and similarly define\footnote{We abuse the notation $\sigma\in\prod_j S_j$ here, meaning that some tuple formed by permuting the elements of $\sigma$ is contained in the Cartesian product.}
$K_2^{(d)}(S_1,\ldots,S_{d+1})$.
Then, one can see that
\begin{equation*}
  \labelcutd_{\ncut{d}}(K_1,K_2) = \frac{\max_{S_j\subseteq X_j:j\in[d+1]}\left||K_1^{(d)}(S_1,\ldots,S_{d+1})|-|K_2^{(d)}(S_1,\ldots,S_{d+1})|\right|}{\vert(K)^{d+1}}.
\end{equation*}
The $d$-dimensional labeled cut-distance describes how different two simplicial complexes may look when observing $d$-simplices across partitions of the node set.

The distances $\labelcutd_{\ncut{d}}$ for $d\geq 1$ are each only able to characterize the difference between two complexes in a single dimension at a time.
We define the labeled cut-distance by taking a weighted sum of all such distances.
Namely, for a nonnegative, summable sequence $(\alpha_j)_{j\geq 1}$ of real numbers, put
\begin{equation*}
  \labelcutd_{\cut}(K_1,K_2;(\alpha_j)) = \sum_{j=1}^\infty \alpha_j\labelcutd_{\ncut{j}}(K_1,K_2).
\end{equation*}
Again, this definition naturally extends to include weighted simplicial complexes as well.

The family of labeled cut-distances are dependent upon the assumption that the node sets of $K_1$ and $K_2$ are equal in size and correspond to one another.
We now consider the case where two simplicial complexes have node sets that do not align.
As a preliminary definition, for a nonnegative, summable sequence $(\alpha_j)_{j\geq 0}$, and two simplicial complexes $K_1,K_2$ such that $\vert(K_1)=\vert(K_2)$, put
\begin{equation*}
  \unlabelcutd_{\cut}(K_1,K_2;(\alpha_j)) = \min_{\phi:\Vert(K_2)\to\Vert(K_1)}\labelcutd_{\cut}(K_1,\phi(K_2);(\alpha_j)),
\end{equation*}
where $\phi$ ranges over bijective maps from $\Vert(K_2)$ to $\Vert(K_1)$.
That is to say, $\unlabelcutd_{\cut}$ first finds the optimal alignment between the two simplicial complexes, and then compares them via $\labelcutd_{\cut}$.

With this in hand, we now extend the definition to handle simplicial complexes whose node sets differ in size.
First, for an integer $m>0$, define the $m$-blowup of a simplicial complex $K$ as the simplicial complex $mK$ where $\Vert(mK)=\Vert(K)\times[m]$, and a set $\{(v_i,j_i)\}_{i=1}^{d+1}\in(mK)^{(d)}$ if and only if $\{v_i\}_{i=1}^{d+1}\in K^{(d)}$, where we assume that $\{v_i\}_{i=1}^{d+1}$ is a set of cardinality $d+1$ (no duplicate nodes).
Then, for two simplicial complex $K_1,K_2$ such that $\vert(K_i)=n_i$, define the unlabeled cut-distance as
\begin{equation*}
  \cutd_{\cut}(K_1,K_2;(\alpha_j)) = \lim_{m\to\infty}\unlabelcutd_{\cut}(mn_2K_1,mn_1K_2;(\alpha_j)).
\end{equation*}
The unlabeled cut-distance blows both simplicial complexes up by a very large factor, so that it can then find a very fine alignment between them, amounting to a so-called fractional overlay of the node sets.
Indeed, $\cutd_{\cut}$ can be equivalently defined in terms of a minimizing fractional overlay of the node sets \citep[see][Eq.~8.10]{Lovasz2012}, but we omit that discussion here.

%% file: sections/30_complexons.tex
\section{Complexons and Random Sampling}
\label{sec:complexons}

As in the theory of graph limits, the convergence of a sequence of simplicial complexes can be defined in terms of homomorphism densities.
Let $K_1,K_2,\ldots$ be a sequence of simplicial complexes.
We say that the sequence $(K_n)_{n\geq 1}$ is convergent if for all simplicial complexes $F$, it holds that $(\thom(F,K_n))_{n\geq 1}$ is a convergent sequence.
The remainder of this paper is devoted to understanding the appropriate limiting objects for such convergent sequences.

The appropriate analogs for graphons in this case are functions on the disjoint union of unit cubes of dimension greater than or equal to $2$.
A complexon is a measurable\footnote{In this paper, ``measurable'' means Borel-measurable. The properties considered are not sensitive to differences on sets of measure zero, so one could also consider the Lebesgue measure.} function
\begin{equation*}
  W:\djunion_{d\geq 1}[0,1]^{d+1}\to[0,1],
\end{equation*}
such that $W$ is measurable and totally symmetric.
That is to say, the restriction of $W$ to each $[0,1]^{d+1}$ is measurable and symmetric in its coordinates.
We find it convenient to assume that $W([0,1])=\{1\}$ and $W()=1$, \ie, $W$ evaluated on the ``empty tuple'' has unit value.

In the same way that a simplicial complex as a collection of sets closed under restriction can be equivalently described by its facets, we formulate a faceted complexon.
For a complexon $W$, the faceted complexon $\faceted{W}$ is a totally symmetric measurable function $\faceted{W}:\djunion_{d\geq 1}[0,1]^{d+1}\to[0,1]$ obeying the following rule.
For all $d\geq 1, (x_1,\ldots,x_{d+1})\in[0,1]^{d+1}$, put
\begin{equation*}
  \faceted{W}(x_1,\ldots,x_{d+1}) = \prod_{\sigma\subseteq[d+1]}W(x_\sigma),
\end{equation*}
where $x_\sigma$ denotes the coordinates of $(x_1,\ldots,x_{d+1})$ indexed by $\sigma$ (this is well-defined, by the total symmetry of the complexon $W$).
Prematurely interpreting the values taken by $W$ as ``probabilities'' of simplices existing conditioned on the existence of their faces, the faceted complexon $\faceted{W}$ intuitively describes the probability of a simplex existing.
For instance, if $W(x_1,x_2,x_3)=1$, but $W(x_1,x_2)=W(x_1,x_3)=W(x_2,x_3)=0.5$, then $\faceted{W}(x_1,x_2,x_3)=0.125$, reflecting the closure of a simplicial complex under restriction.

\subsection{Complexons as Random Simplicial Complex Models}
\label{sec:complexons:sampling}

A complexon induces a distribution of random simplicial complexes via a simple sampling procedure.
Let $W$ be a complexon, and $S_n=\{x_1,\ldots,x_n\}$ be a set of $n$ points contained in $[0,1]$, for some integer $n\geq 1$.
From $W$ and $S_n$, define a weighted simplicial complex $H=W[S_n]$ so that $\Vert(H)=[n], H=2^{[n]},$ and each $\sigma\in H$ has weight $\omega(\sigma)=W(x_\sigma)$.
From the weighted simplicial complex $H$, we randomly sample an unweighted simplicial complex in the following way.
Let $K$ be a simplicial complex, and put $\Vert(K)=[n]$.
Inductively for $d\geq 1$, for each $d+1$-subset $\sigma\subseteq[n]$ such that every strict subset of $\sigma$ is contained in $K$, include $\sigma$ in $K$ with probability $\beta(\sigma)$ (\ie, the weight of the simplex $\sigma$ in $H$).
Upon termination, this yields a finite simplicial complex $K$.
We call the distribution from which $K$ is drawn $\sample{K}(H)$.

Noting that $H$ is defined based on the complexon $W$ via the set $S_n$, we randomize the set $S_n$ to yield a more general random simplicial complex model.
Let $S_n$ be such that $x_1,\ldots,x_n$ are \iid{} samples from the uniform distribution probability distribution on $[0,1]$.
Then, the random simplicial complex sampled from $W[S_n]$ with random $S_n$ is denoted $\sample{K}(n,W)$.

The distribution $\sample{K}(n,W)$ can be thought of in the same way as the lower model of~\citet{Farber2022}.
Indeed, it can be constructed by taking the largest simplicial complex contained in a random hypergraph on node set $[n]$ whose hyperedges $\sigma\subseteq[n]$ exist independently each with probability $W(x_\sigma)$.
Moreover, the faceted complexon $\faceted{W}$ can be used via the upper model of~\citet{Farber2022} to yield the distribution $\sample{K}(n,W)$.
Again, take a random hypergraph $H$ on the node set $[n]$ whose hyperedges $\sigma\subseteq[n]$ exist independently each with probability $\faceted{W}(x_\sigma)$.
Taking $K$ to be the smallest simplicial complex that contains $H$, it can be shown that $K$ is distributed according to $\sample{K}(n,W)$.
We go in to more detail on the relationship between complexons and hypergraph limits in \cref{sec:remarks:hypergraphs}.

\subsection{Homomorphisms in Complexons}

Motivated by the random sampling model $\sample{K}(n,W)$, the homomorphism densities $\thom$ and $\tind$ can be naturally generalized to complexons.
Let $F$ be a simplicial complex, and identify $\Vert(F)=[n]$ for some $n\geq 1$.
The homomorphism density of $F$ in a complexon $W$ asks the following question: what is the probability that $F$ is contained in the random simplicial complex $\sample{K}(n,W)$?
Based on this question, we directly define the homomorphism density of $F$ in a complexon $W$.
Put
\begin{equation}\label{eq:thom-graphon}
  \thom(F,W) = \int_{[0,1]^{\Vert(F)}}\prod_{\sigma\in F}W(x_\sigma)\drv{x}.
\end{equation}
The question of the homomorphism density can also be posed in terms of facets.
Put
\begin{equation}\label{eq:thom-faceted}
  \thom(\faceted{F},\faceted{W}) = \int_{[0,1]^{\Vert(F)}}\prod_{\sigma\in \faceted{F}}\faceted{W}(x_\sigma)\drv{x}.
\end{equation}
A simple comparison of the integrands of \eqref{eq:thom-graphon} and \eqref{eq:thom-faceted} yields the following result, further justifying the name ``faceted complexon'' for $\faceted{W}$.

\begin{lemma}\label{lemma:thom-equiv}
  For any simplicial complex $F$ and complexon $W$, it holds that $\thom(F,W)=\thom(\faceted{F},\faceted{W})$.
\end{lemma}

In the case of graphons, the induced homomorphism density is computed in a fashion similar to \eqref{eq:thom-graphon}, but with a product over all antifacets included in the integrand.
For simplicial complexes, taking the product over all non-simplices would be redundant, as the non-existence of an edge, for instance, implies the non-existence of all higher-order simplices that would contain that edge.
Treating a graph with no isolated nodes as a simplicial complex of dimension one, notice that the facets of a graph are its edges, and the antifacets are its non-edges.
Based on this, we write the induced homomorphism density of $F$ in $W$ as follows:
\begin{equation*}
  \tind(F,W) = \int_{[0,1]^{\Vert(F)}}\prod_{\sigma\in F}W(x_\sigma)\prod_{\tau\in\antifaceted{F}}\left(1-W(x_\tau)\right)\drv{x}.
\end{equation*}
The following result can be shown by unrolling definitions, in order to relate the induced homomorphism densities and random samples of a complexon:
\begin{lemma}\label{lemma:ind-sample}
  For a complexon $W$ and a simplicial complex $F$ such that $\Vert(F)=[n]$, it holds that
  \begin{equation*}
    \tind(F,W) = \Pr\left\{\sample{K}(n,W)=F\right\}.
  \end{equation*}
\end{lemma}

In graph theory, we are typically concerned with properties of graphs that do not change under the relabeling of nodes.
Similarly, in graph limit theory, we typically describe properties of complexons up to measure-preserving transformations of $[0,1]$.
For complexons, a property of this type holds for homomorphism densities.
Let $\phi:[0,1]\to[0,1]$ be a measure-preserving transformation, and for a complexon $W$, define
\begin{equation*}
  W^\phi(x_1,\ldots,x_{d+1}) = W_2(\phi(x_1),\ldots,\phi(x_{d+1}))
\end{equation*}
for all $d\geq 1$ and $x\in[0,1]^{d+1}$.
The following result proceeds directly from the definition of the (induced) homomorphism density.

\begin{lemma}\label{lemma:thom-transform}
  Let $W$ be a complexon, and $\phi:[0,1]\to[0,1]$ a measure-preserving transformation.
  Then, for all simplicial complexes $F$,
  \begin{align*}
    \thom(F,W) &= \thom(F,W^\phi) \\
    \tind(F,W) &= \tind(F,W^\phi).
  \end{align*}
\end{lemma}

\subsection{The Cut Distance}

Similar to the family of cut-distances for simplicial complexes, we define cut-distances for complexons.
Let $W_1,W_2$ be complexons, and $d\geq 1$ some integer.
The $d$-dimensional labeled cut-distance between $W_1$ and $W_2$ is defined as follows:
\begin{equation*}
  \labelcutd_{\ncut{d}}(W_1,W_2) = \sup_{S_1,\ldots,S_{d+1}\in\Borel[0,1]}\left|\int_{\prod_j S_j}\left(W_1(x)-W_2(x)\right)\drv{x}\right|,
\end{equation*}
where $\Borel[0,1]$ denotes the Borel $\sigma$-field on the interval $[0,1]$.
This relates to the cut-norm defined for general measurable functions $U:\djunion_{d\geq 1}[0,1]^{d+1}\to\reals$:
\begin{equation*}
  \|U\|_{\ncut{d}} =
  \sup_{S_1,\ldots,S_{d+1}\in\Borel[0,1]}\left|\int_{\prod_j S_j}U(x)\drv{x}\right|,
\end{equation*}
so that $\labelcutd_{\ncut{d}}(W_1,W_2)=\|W_1-W_2\|_{\ncut{d}}$.
Alternatively, the cut-norm can be written as
\begin{equation*}
  \|U\|_{\ncut{d}} =
  \sup_{f_1,\ldots,f_{d+1}:[0,1]\to[0,1]}\left|\int_{[0,1]^{d+1}}U(x)f_1(x_1)\cdots f_{d+1}(x_{d+1})\drv{x}\right|,
\end{equation*}
where $f_1,\ldots,f_{d+1}$ are taken to be measurable functions from $[0,1]$ to $[0,1]$.

Following \citep[Eq.~7.2]{Borgs2008}, restricting the sets considered in the $d$-dimensional labeled cut-distance to be pairwise disjoint only changes the distance by a constant factor depending on $d$.
\begin{lemma}\label{lemma:disjoint-cut}
  Let $W_1,W_2\in\kerns$ be complexons, and let $d\geq 1$.
  Put $\mathcal{S}$ as the collection of pairwise disjoint sets $S_1,\ldots,S_{d+1}\in\Borel[0,1]$.
  Then,
  \begin{equation*}
    \sup_{(S_1,\ldots,S_{d+1})\in\mathcal{S}}\left|\int_{\prod_j S_j}(W_1(x)-W_2(x))\drv{x}\right|
    \leq
    \frac{1}{(d+1)^{d+1}}\labelcutd_{\ncut{d}}(W_1,W_2).
  \end{equation*}
\end{lemma}
The proof follows the presentation of \citet[Lemma~E.2]{Janson2013}.
\begin{proof}
  Let $\partition{A}=\{A_i\}_{i=1}^n$ be an $n$-equipartition of $[0,1]$ (that is, a partition of $[0,1]$ into $n$ measurable subsets all with measure $1/n$).
  Let $I$ be an element of $[d+1]^n$ chosen uniformly at random.
  For each $j\in[d+1]$, define $B_j=\bigcup_{m:I(m)=j}A_m$.

  For any measurable sets $S_1,\ldots,S_{d+1}$, the sets $S_1\cap B_1,\ldots,S_{d+1}\cap B_{d+1}$ are pairwise disjoint, so that
  \begin{equation}\label{eq:disjoint-cut-bound1}
    \Ex\left[\left|\int_{\prod_j S_j\cap B_j}(W_1(x)-W_2(x))\drv{x}\right|\right]
    \leq
    \sup_{(S_1,\ldots,S_{d+1})\in\mathcal{S}}\left|\int_{\prod_j S_j}(W_1(x)-W_2(x))\drv{x}\right|.
  \end{equation}
  Moreover,
  \begin{equation}\label{eq:disjoint-cut-bound2}
    \Ex\left[\int_{\prod_j S_j\cap B_j}(W_1(x)-W_2(x))\drv{x}\right] =
    \sum_{i_1\neq\ldots\neq i_{d+1}\in[n]}\int_{\prod_j S_j\cap A_{i_j}}\frac{W_1(x)-W_2(x)}{(d+1)^{d+1}}\drv{x}.
  \end{equation}
  One can see that the right-hand side of \eqref{eq:disjoint-cut-bound2} approaches the quantity
  \begin{equation*}
    \frac{1}{(d+1)^{d+1}}\int_{\prod_j S_j}(W_1(x)-W_2(x))\drv{x}
  \end{equation*}
  as $n\to\infty$.
  Combining \eqref{eq:disjoint-cut-bound1} and \eqref{eq:disjoint-cut-bound2} concludes the proof.
\end{proof}

As the labeled $d$-dimensional cut-distance only considers the value of complexons on $[0,1]^{d+1}$, we define a cut-distance that considers all dimensions.
Let $(\alpha_j)_{j\geq 1}$ be a nonnegative, summable sequence of real numbers.
For complexons $W_1,W_2$, put
\begin{equation*}
  \labelcutd_{\cut}(W_1,W_2;(\alpha_j)_{j\geq 1}) = \sum_{j\geq 1}\alpha_j\labelcutd_{\ncut{j}}(W_1,W_2).
\end{equation*}
One can clearly see that for any complexons $W_1,W_2$,
\begin{equation}\label{eq:norm-domination}
  \labelcutd_{\cut}(W_1,W_2;(\alpha_j)_{j\geq 1}) \leq
  \sum_{j\geq 1}\alpha_j\int_{[0,1]^{j+1}}|W_1(x)-W_2(x)|\drv{x}.
\end{equation}
We extend this definition to finite nonnegative sequences $\alpha_1,\ldots,\alpha_d$ by implicitly assuming that $\alpha_j=0$ for $j>d$.

The labeled cut-distance implicitly assumes a correspondence between the domains of $W_1$ and $W_2$.
Intuitively, this corresponds to coupling the distributions $\sample{K}(n,W_1)$ and $\sample{K}(n,W_2)$ (for any $n\geq 1$) in such a way that the points $x_1,\ldots,x_n\in[0,1]$ are always equal for $\sample{K}(n,W_1)$ and $\sample{K}(n,W_2)$.

To decouple the domains, we define the unlabeled cut-distance between $W_1$ and $W_2$ by taking the infimum over measure preserving transformations of the domain.
Put, again for a (finite or infinite) nonnegative summable sequence $(\alpha_j)_{j\geq 1}$,
\begin{equation*}
  \cutd_{\cut}(W_1,W_2;(\alpha_j)_{j\geq 1}) =
  \inf_{\phi:[0,1]\to[0,1]}\labelcutd_{\cut}(W_1,W_2^\phi;(\alpha_j)_{j\geq 1}),
\end{equation*}
where $\phi$ ranges over measure-preserving transformations.
Similar to \eqref{eq:norm-domination}, we have the inequality
\begin{equation*}
  \cutd_{\cut}(W_1,W_2;(\alpha_j)_{j\geq 1}) \leq
  \inf_{\phi:[0,1]\to[0,1]}\sum_{j\geq 1}\alpha_j\int_{[0,1]^{j+1}}|W_1(x)-W_2^\phi(x)|\drv{x}.
\end{equation*}

The definition of the unlabeled cut-distance for simplicial complexes in terms of the limiting distance under blowups motivates an immediate representation of simplicial complexes as complexons, closely resembling the representation of graphs as graphons via ``pixel pictures.''
For a simplicial complex $K$ whose nodes are identified as $\Vert(K)=[n]$, define a complexon $W_K$ in the following way.
Define the sets $P_j=\hintCO{(j-1)/n,j/n}$ for $j\in[n]$.
Then, for each $\{j_1,\ldots,j_{d+1}\}\in K$, put $W(x_1,\ldots,x_{d+1})=1$ for all permutations of $(x_1,\ldots,x_{d+1})\in\prod_{\ell=1}^{d+1}P_{j_\ell}$.
Otherwise, put $W_K(x_1,\ldots,x_{d+1})=0$.
If instead of a simplicial complex $K$ we have a weighted simplicial complex $(H,\omega)$, we put $W_H(x_1,\ldots,x_{d+1})=\omega(\{j_1,\ldots,j_{d+1}\})$.

For a simplicial complex $K$ and a complexon $W$, we will find it convenient to use the notation $\labelcutd_{\cut}(K,W)$ and $\cutd_{\cut}(K,W)$ to refer to $\labelcutd_{\cut}(W_K,W)$ and $\cutd_{\cut}(W_K,W)$, respectively.
Moreover, the complexon corresponding to a simplicial complex preserves the cut-distance and all homomorphism densities.

\begin{lemma}\label{lemma:finite-to-complexon-cut}
  For any simplicial complexes $F,K$ and any nonnegative, summable sequence $(\alpha_j)_{j\geq 1}$,
  \begin{align*}
    \thom(F,K) &= \thom(F,W_K) \\
    \cutd_{\cut}(F,K;(\alpha_j)) &= \cutd_{\cut}(W_F,W_K;(\alpha_j)).
  \end{align*}
  Moreover, if $\Vert(F)=\Vert(K)$, then, under any identification of $\Vert(F)$ with $[n]$,
  \begin{equation*}
    \labelcutd_{\cut}(F,K;(\alpha_j)) = \labelcutd_{\cut}(W_F,W_K;(\alpha_j)).
  \end{equation*}
  Furthermore, if $(H,\omega)$ is a weighted simplicial complex such that $\Vert(F)=\Vert(H)$, then
  \begin{equation*}
    \labelcutd_{\cut}(F,H;(\alpha_j)) = \labelcutd_{\cut}(W_F,W_H;(\alpha_j)).
  \end{equation*}
\end{lemma}

%% file: sections/41_counting-lemma.tex
\subsection{The Counting Lemma}

We begin by showing how closeness in the cut-distance yields closeness in the sense of homomorphism densities.
\begin{lemma}[Counting Lemma]\label{lemma:counting}
  Let $F$ be a simplicial complex.
  Define the sequences $\alpha_j=|(\faceted{F})^{(j)}|, \beta_j=|F^{(j)}|, \gamma_j=|F^{(j)}|+|(\antifaceted{F})^{(j)}|$ for $j\geq 1$.
  Then, for complexons $U,W$, the following three inequalities hold:
  \begin{align*}
    \left|\thom(F,U)-\thom(F,W)\right| &\leq \cutd_{\cut}(\faceted{U},\faceted{W};(\alpha_j)_{j\geq 1}) \\
    \left|\thom(F,U)-\thom(F,W)\right| &\leq \cutd_{\cut}(U,W;(\beta_j)_{j\geq 1}) \\
    \left|\tind(F,U)-\tind(F,W)\right| &\leq \cutd_{\cut}(U,W;(\gamma_j)_{j\geq 1}).
\end{align*}
\end{lemma}

This result is analogous to \citep[Lemma~4.1]{Lovasz2006}, with a similar proof as well.

\begin{proof}
  We establish the first inequality, as the other two follow from a similar argument.
  Order the elements of $\faceted{F}$ as $\{\sigma_1,\ldots,\sigma_m\}$, where $m=|\faceted{F}|$.
  For $x\in[0,1]^{\Vert(F)}, t\in[m]$, define
  \begin{equation*}
    X_t(x) =
    \left(\prod_{s<t}\faceted{U}(x_{\sigma_s})\right)\cdot
    \left(\prod_{s>t}\faceted{W}(x_{\sigma_s})\right)\cdot
    \left(\faceted{U}(x_{\sigma_t})-\faceted{W}(x_{\sigma_t})\right).
  \end{equation*}
  One can check that
  \begin{align*}
    \left|\thom(F,U)-\thom(F,W)\right| &= \left|\int_{[0,1]^{\Vert(F)}}\sum_{t=1}^m X_t(x)\drv{x}\right| \\
                               &= \left|\sum_{t=1}^m\int_{[0,1]^{\Vert(F)\setminus\sigma_t}}\int_{[0,1]^{\sigma_t}}X_t(x)\prod_{j\in\sigma_t}\drv{x_j}\prod_{j\in \Vert(F)\setminus\sigma_t}\drv{x_j}\right|.
  \end{align*}
  For all $t\in[m]$, it holds that
  \begin{equation*}
    \prod_{s<t}\faceted{U}(x_{\sigma_s})\prod_{s>t}\faceted{W}(x_{\sigma_s})\in[0,1],
  \end{equation*}
  This implies via \cref{lemma:thom-equiv} that
  \begin{align*}
    \left|t(F,U)-t(F,W)\right| &\leq \left|\sum_{t=1}^m\labelcutd_{\ncut{\dim{\sigma_t}}}(\faceted{U},\faceted{W})\right| \\
                               &= \labelcutd_{\cut}(\faceted{U},\faceted{W};(\alpha_j)_{j\geq 1}).
  \end{align*}
  Taking the infimum over measure-preserving transformations $W^\phi$ yields the first inequality, via \cref{lemma:thom-transform}.
  The proofs of the second and third inequalities proceed similarly.
\end{proof}

This result is perhaps not too surprising.
Thinking of the cut-metric as describing the global differences between complexons, closeness globally forces closeness locally (that is, in the sense of homomorphism densities).
More formally, for any simplicial complex $F$, the homomorphism density $\thom(F,\cdot):\kerns\to[0,1]$ is continuous with respect to the cut-distance $\cutd_{\cut}(\cdot,\cdot;(\alpha_j)_{j\geq 1})$ for any strictly positive, summable sequence $(\alpha_j)_{j\geq 1}$.
Similarly, since $\tind(F,W)=\Pr\left\{\sample{K}(n,W)=F\right\}$ for any simplicial complex $F$ with $\Vert(F)=[n]$, closeness in the cut-metric forces two complexons to yield similar random sampling models, particularly for small $n$.

%% file: sections/42_sampling-lemma.tex
\subsection{The Sampling Lemma}

Before proving the inverse of \cref{lemma:counting}, we establish a concentration result for samples $\sample{K}(n,W)$: namely, we show that sufficiently large samples drawn from a complexon will be close in the cut-distance.
Formally,
\begin{lemma}[Sampling Lemma]\label{lemma:sampling}
  Let $W$ be a complexon.
  Let $n\geq 1$ be an integer, and $\alpha_1,\ldots,\alpha_d$ be a finite nonnegative sequence.
  Then, with probability at least $1-\exp\left(-n/(2\log_2{n})\right)$, it holds that
  \begin{equation*}
    \cutd_{\cut}(\faceted{W},\sample{K}(n,W);(\alpha_j))\leq
    \frac{8\cdot 2^d+1}{\sqrt{\log_2{n}}}\sum_{j=1}^d\alpha_j.
  \end{equation*}
\end{lemma}
We leave the proof to \cref{sec:proofs:sampling}.
As a corollary of \cref{lemma:sampling}, it can be shown via the Borel-Cantelli Lemma that every faceted complexon $\faceted{W}$ arises as the limit of a sequence of simplicial complexes.
That is,
\begin{corollary}\label{coro:sampling-convergence}
  Let $\alpha_1,\ldots,\alpha_d$ be a finite nonnegative sequence.
  For a complexon $W$, the sequence of simplicial complexes $(\sample{K}(n,W))_{n\geq 1}$ converges to $\faceted{W}$ in the cut-distance $\cutd_{\cut}(\cdot,\cdot;(\alpha_j))$ with probability $1$.
\end{corollary}

%% file: sections/43_inverse-counting-lemma.tex
\subsection{The Inverse Counting Lemma}

We are now ready to state the inverse of \cref{lemma:counting}.

\begin{lemma}[Inverse Counting Lemma]\label{lemma:inverse-counting}
  Let $U,W$ be two complexons, and suppose that for some $d\geq 1, n\geq 2$, for all simplicial complexes $F$ on $n$ nodes of dimension $d$,
  \begin{equation*}
    \left|\thom(F,U)-\thom(F,W)\right|\leq
    0.999\cdot 2^{-((1+n)^{d+2})}
  \end{equation*}
  Then, for all finite nonnegative sequences $\alpha_1,\ldots,\alpha_d$,
  \begin{equation*}
    \cutd_{\cut}(\faceted{U},\faceted{W};(\alpha_j))\leq
    \frac{2^{d+4}+2}{\sqrt{\log_2{n}}}\sum_{j=1}^d\alpha_j.
  \end{equation*}
\end{lemma}

We leave the proof to \cref{sec:proofs:inverse-counting}.
The bound in \cref{lemma:inverse-counting} is only useful for $\log_2{n}\gtrsim 2^{d+4}$, with homomorphism densities that are extremely close.
However, it does demonstrate that the cut-distance, a quantity determined by a Lebesgue integral on a continuous domain over all possible Borel sets of certain dimension, can be characterized by a collection of quantities determined by finite objects.
Keeping in mind that this bound is generic, holding for any possible pair of complexons, one would expect it to be much better if considered on a tamer family of complexons, \eg, those that are stepfunctions on some coarse partition of $[0,1]$.
Indeed, we have the following result for stepfunctions.
\begin{proposition}\label{prop:step-inverse-counting}
  Let $U,W$ be two complexons.
  Suppose for some integer $m\geq 1$, $U$ and $W$ are both stepfunctions on respective equipartitions $\partition{P}_U,\partition{P}_W$, each with at most $m$ steps.
  Furthermore, suppose that for some $d\geq 1, n\geq 2$, for all simplicial complexes $F$ on $n$ nodes of dimension $d$,
  \begin{equation*}
    |\thom(F,U)-\thom(F,W)| \leq 0.999\cdot 2^{-((1+n)^{d+2})}.
  \end{equation*}
  Then, for all finite nonnegative sequences $\alpha_1,\ldots,\alpha_d$,
  \begin{equation*}
    \cutd_{\cut}(\faceted{U},\faceted{W};(\alpha_j))\leq
    \frac{\sqrt{m}(4d+5)+4}{\sqrt{n}}\sum_{j=1}^d\alpha_j.
  \end{equation*}
\end{proposition}
We omit the proof, as the argument is a mere simplification of the proofs of \cref{lemma:sampling,lemma:inverse-counting}.
By making the stronger assumption that the complexons of interest are stepfunctions on coarse equipartitions of $[0,1]$, \cref{prop:step-inverse-counting} presents an inverse counting lemma that holds for $n\gtrsim md^2$.

%% file: sections/44_compactness.tex
\subsection{Topology of the Space of Complexons}
\label{sec:space:compact}

Denote the set of all complexons by $\kerns$, and let $(\alpha_j)_{j\geq 1}$ be a nonnegative, summable sequence of real numbers.
If two complexons $W_1,W_2\in\kerns$ only differ on a set of measure zero, then $\labelcutd_{\cut}(W_1,W_2;(\alpha_j)_{j\geq 1})=0$.
Indeed, the labeled cut-distance is only a pseudometric on $\kerns$.
Furthermore, the unlabeled cut-distance $\delta_{\cut}$ can take value zero if there exists a measure-preserving bijection $\phi:[0,1]\to[0,1]$ such that $W_1$ and $W_2^\phi$ only differ on a set of measure zero, and is thus also a pseudometric.
The following result states that, topologically speaking, the choice of $(\alpha_j)$ in this construction does not matter too much.

\begin{lemma}\label{lemma:topological-equivalence}
  Let $(\alpha_j)_{j\geq 1}$ and $(\beta_j)_{j\geq 1}$ be strictly positive, summable sequences of real numbers.
  Then, the cut-metrics $\delta_{\cut}(\cdot,\cdot;(\alpha_j)_{j\geq 1})$ and $\delta_{\cut}(\cdot,\cdot;(\beta_j)_{j\geq 1})$ are topologically equivalent pseudometrics on $\kerns$.
\end{lemma}

With \cref{lemma:topological-equivalence,lemma:counting,lemma:inverse-counting} in mind, we endow $\kerns$ with a topology determined by any strictly positive, summable sequence $(\alpha_j)_{j\geq 1}$.
In particular, we define the canonical topology on $\kerns$ as the topology determined by the pseudometric
\begin{equation*}
  \weakcutd(U,W;(\alpha_j)_{j\geq 1}) = \cutd_{\cut}(\faceted{U},\faceted{W};(\alpha_j)_{j\geq 1}),
\end{equation*}
for an arbitrary strictly positive, summable sequence $(\alpha_j)_{j\geq 1}$.
We also find it useful to define the $\cutd$-topology on $\kerns$ as the topology determined by the pseudometric $\cutd_{\cut}$, again for some strictly positive, summable sequence.
The following result about the canonical topology on $\kerns$ is immediate.
\begin{proposition}\label{prop:simple-dense}
  Let $\fkerns\subset\kerns$ be the set of complexons that arise from simplicial complexes, \ie{}, for every $W\in\fkerns$, there is some simplicial complex $K$ such that $W=W_K$.
  Then, $\fkerns$ is dense in $\kerns$ in the canonical topology.
\end{proposition}
Our primary result of this section establishes the compactness of $\kerns$ with respect to this topology.
\begin{theorem}[Compactness]\label{thm:compactness}
  The space $\kerns$ with the canonical topology is compact.
\end{theorem}
We leave the proofs of \cref{lemma:topological-equivalence,prop:simple-dense,thm:compactness} to \cref{sec:proofs:compactness}.
So far, we have established that the space of complexons with the canonical topology is ``essentially finite,'' in that it can be approximated by finitely many (\cref{thm:compactness}) finite simplicial complexes (\cref{prop:simple-dense}) in the cut-metric.

The canonical topology on $\kerns$ can also be defined based on homomorphism densities.
\begin{theorem}[Equivalence of Cuts and Homomorphisms]\label{thm:hom-cut-eq}
  Let $(a_j)_{j\geq 1}$ be a strictly positive, summable sequence of real numbers.
  Noting that there are countably many simplicial complexes (up to isomorphism), let $(F_j)_{j\geq 1}$ be an enumeration of the set of all isomorphism classes of simplicial complexes, making the identification $\Vert(F_j)=[\vert(F_j)]$ for each $j\geq 1$.
  Define a pseudometric $\rho$ on $\kerns$ so that for any $U,W\in\kerns$,
  \begin{equation*}
    \rho(U,W;(a_j)) = \sum_{j=1}^\infty a_j\cdot|\thom(F_j,U)-\thom(F_j,W)|.
  \end{equation*}
  Then, the topology on $\kerns$ induced by $\rho$ is equal to the canonical topology.
\end{theorem}

\begin{proof}
  We prove that a sequence of complexons $W_1,W_2,\ldots$ is convergent in the canonical topology if and only if it is convergent with respect to the pseudometric $\rho$.

  \medskip\noindent\textit{(If).}
  Let $\epsilon>0$ be given, and suppose $W_m\to W$ in the canonical topology.
  Let $M$ be such that $\sum_{j>M}a_j<\epsilon/2$.
  Then, by \cref{lemma:counting}, there exists some $m_0$ such that for all $j\leq M, m>m_0$, it holds that
  \begin{equation*}
    \left|\thom(F_j,W_m)-\thom(F_j,W)\right|<\frac{\epsilon}{2M\max_{k\leq M}a_k}.
  \end{equation*}
  It follows, then, that for all $m>m_0$, $\rho(W_m,W;(a_j)_{j\geq 1})<\epsilon$.
  Since $\epsilon$ was given arbitrarily, this implies that $W_m\to W$ in the topology induced by $\rho$, as desired.

  \medskip\noindent\textit{(Only if).}
  Let $\epsilon>0$ be given, and let $(\alpha_j)_{j\geq 1}$ be an arbitrary strictly positive, summable sequence.
  Suppose $W_m\to W$ in the topology induced by $\rho$.
  Let $d$ be such that $\sum_{j>d}\alpha_j<\epsilon/2$.
  Put $n$ such that
  \begin{equation*}
    \frac{2^{d+4}+2}{\sqrt{\log_2{n}}}\sum_{j=1}^{d}\alpha_j < \frac{\epsilon}{2}.
  \end{equation*}
  Noting that there are only finitely many isomorphism classes of simplicial complexes on $n$ nodes of dimension $d$, let $M$ be such that for every simplicial complex $F$ on $n$ nodes of dimension $d$, there is some $j\leq M$ such that $F\cong F_j$.
  By assumption, there is some $m_0$ such that for all $j\leq M, m>m_0$
  \begin{equation*}
    |\thom(F_j,W_m)-\thom(F_j,W)| < 0.999\cdot 2^{-((1+n)^{d+2})}.
  \end{equation*}
  By \cref{lemma:inverse-counting}, this implies that
  \begin{equation*}
    \cutd_{\cut}(\faceted{W}_m,\faceted{W};(\alpha_j)_{j\geq 1})
    \leq\frac{2^{d+4}+2}{\sqrt{\log_2{n}}}\sum_{j=1}^{d}\alpha_j + \sum_{j>d}\alpha_j
    <\epsilon.
  \end{equation*}
  Since $\epsilon$ was given arbitrarily, this implies that $W_m\to W$ in the canonical topology, as desired.
\end{proof}

Observe that the pseudometric $\rho$ in \cref{thm:hom-cut-eq} is defined for an arbitrary strictly positive, summable sequence $(a_j)_{j\geq 1}$ and enumeration $(F_j)_{j\geq 1}$: this indicates that, similar to the $\cutd$-topology as described by \cref{lemma:topological-equivalence}, the topology induced by $\rho$ is not dependent on the particular sequence $(a_j)_{j\geq 1}$ and enumeration $(F_j)_{j\geq 1}$.
In other words, there is a ``hom-topology'' on $\kerns$ that is pseudometrizable by any such $\rho$.
Moreover, one can see that a sequence $(W_m)_{m\geq 1}$ converges in the hom-topology if and only if for all simplicial complexes $F$, $\thom(F,W_m)$ is a convergent sequence.
\Cref{thm:hom-cut-eq}, then, indicates that the canonical topology on $\kerns$ is the topology defined by convergence in homomorphism densities.

Indeed, by the definition of the canonical topology on $\kerns$, two complexons $U,W$ are identified with each other if and only if they are indistinguishable \emph{by sampling}, which is a property that holds modulo sets of measure zero across the faceted complexons, rather than the original complexons themselves.
The identification of complexons by the canonical topology not only ignores differences on sets of measure zero, which have no consequence in sampling and homomorphism densities, but also enforces a ``consistency structure across different dimensions'' as suggested by~\citet{Bobrowski2022}, by considering the faceted version.

%% file: sections/51_posets.tex
\subsection{Limits of partially ordered sets}

Limits of partially ordered sets (posets) in homomorphism density were considered by \citet{Janson2011}, where it was shown that any sequence of partially ordered sets whose homomorphism densities converge have a limit representable by a kernel on an ordered probability space.
It was conjectured that all such kernels can be taken to be defined on the ordered probability space $[0,1]$ with the Lebesgue measure, which was affirmed by \citet{Hladky2015}.
In the context of this work, one may wonder if limits of sequences of simplicial complexes could be understood as posets.
Indeed, a simplicial complex $K$ carries with it a strict partial order determined by set inclusion, yielding the partially ordered set $(K,\subset)$.
Moreover, since the simplices have a notion of dimension, there is a rank function associated with $(K,\subset)$, yielding a graded poset $(K,\subset,\dim)$.

A simplicial homomorphism from a simplicial complex $F$ into $K$ is a poset homomorphism as considered by \citet{Janson2011}, but has some additional structure.
Namely, a simplicial homomorphism $\phi:\Vert(F)\to\Vert(K)$ is rank-preserving as a poset homomorphism, in that $\dim(\sigma)=\dim(\phi(\sigma))$ for all $\sigma\in F$.
General poset homomorphisms as considered by \citet{Janson2011} do not require the rank to be preserved.
It is interesting to see how the rank-preservation condition on these poset homomorphisms yields a significantly different limit structure, namely a complexon, as opposed to a kernel on an ordered probability space.
Indeed, finite graded posets may be treated as abstract cell complexes, of which abstract simplicial complexes are a special type -- this suggests that the symmetry conditions of complexons may be relaxed in order to yield limit objects for more general combinatorial structures.

%% file: sections/52_local-random.tex
\subsection{Local Consistent Random Simplicial Complexes}

We have studied the limits of dense sequences of simplicial complexes via complexons, but this is not the only obvious limiting structure.
Another such object for the graph case is a locally consistent random graph model~\citep[Chapter~11.2]{Lovasz2012}.
Analogous structures exist for simplicial complexes.
A random complex model is a sequence $(\mu_n)_{n\geq 1}$ of probability measures such that, for all $n\geq 1$, $\mu_n$ is a probability measure on the set of simplicial complexes $K$ such that $\Vert(K)=[n]$, and if $K$ and $K'$ are isomorphic, then $\mu_n(K)=\mu_n(K')$.

The random complex model $(\mu_n)$ is said to be consistent if for any simplicial complex $K$ on node $[n-1]$, we have
\begin{equation*}
  \mu_{n-1}(K) = \mu_n\left(\left\{F:\Vert(F)=[n],F'=K\right\}\right),
\end{equation*}
where $F'=K$ indicates that the removal of the node $n$ from $F$ (and all simplices containing $n$) yields the simplicial complex $F$.
Furthermore, if for all $n>1$, and disjoint subsets $S,T\subseteq[n]$, the random simplicial complexes on nodes $[|S|]$ and $[|T|]$ determined by taking the induced subcomplex on $S,T$ of a simplicial complex sampled following $\mu_n$ are independent, we say that $(\mu_n)$ is local.

Indeed, any simplicial complex $K$ yields a random complex model.
Define, for $n\geq 1$, the probability measure $\mu_{K,n}$ such that for any simplicial complex $F$ with $\Vert(F)=[n]$
\begin{equation*}
  \mu_{K,n}(F) = \tind(F,K).
\end{equation*}

The following result, following \citep[Theorem~11.7]{Lovasz2012}, links convergent sequences of simplicial complexes to local consistent random complex models.
\begin{theorem}
  Let a convergent sequence of simplicial complexes $K_1,K_2,\ldots$ with $\vert(K_m)\to\infty$ as $m\to\infty$ be given.
  Define, for all $n\geq 1$ and all simplicial complexes $F$ with $\Vert(F)=[n]$, the probability measure
  \begin{equation*}
    \mu_n(F) = \lim_{m\to\infty}\mu_{K_m,n}(F),
  \end{equation*}
  noting that the sequence $K_1,K_2,\ldots$ being convergent implies the limit exists.
  Then, the sequence $(\mu_n)_{n\geq 1}$ forms a local consistent random complex model.
  Conversely, every local consistent random complex model arises this way.
\end{theorem}
The proof follows that of \citep[Theorem~11.7]{Lovasz2012} almost exactly, so we omit it here.
Furthermore, an initial connection with complexons is stated in the following result.
\begin{proposition}\label{prop:lccm-complexon}
  Let $W$ be a complexon.
  Define, for all $n\geq 1$, the probability measure for simplicial complexes $F$ on nodes $[n]$
  \begin{equation*}
    \mu_n(F) = \Pr\left\{\sample{K}(n,W)=F\right\} = \tind(F,W).
  \end{equation*}
  Then, the sequence $(\mu_n)_{n\geq 1}$ forms a local consistent random complex model.
\end{proposition}

Viewing the sampled simplicial complexes from a complexon as being equivalently sampled from a local consistent random complex model is a helpful perspective for establishing convergence of a complexon's samples in the cut-metric.
In particular,
\begin{proposition}\label{prop:lccm-convergent}
  Let $(\mu_n)_{n\geq 1}$ be a local consistent random complex model.
  For each $n\geq 1$, independently draw a simplicial complex $K_n$ on nodes $[n]$ according to the distribution $\mu_n$.
  Then, with probability $1$,
  \begin{enumerate}[label={(\roman*)}]
  \item The sequence $(K_n)$ is convergent
  \item For all $n$, $\lim_{m\to\infty}\mu_{K_m,n} = \mu_{n}$.
  \end{enumerate}
  In particular, if $(\mu_n=\sample{K}(n,W))_{n\geq 1}$ for some complexon $W$, we have that a sequence of increasingly large simplicial complexes sampled from a complexon converges to $\faceted{W}$ with probability $1$.
\end{proposition}
We also omit the proof, as it resembles that of \citep[Lemma~11.8]{Lovasz2012} with very little modification.
\Cref{prop:lccm-complexon,prop:lccm-convergent} together indicate that for a given complexon $W$ and a sequence of increasingly large simplicial complexes $(K_n\sim\sample{K}(n,W))_{n\geq 1}$, $K_n\to\faceted{W}$ with probability $1$.
That is to say, we can (almost surely) generate a sequence converging to an arbitrary complexon by simply taking a sequence of large simplicial complexes sampled from it.
Indeed, this recovers \cref{coro:sampling-convergence}.

\subsubsection{Homogenous random complexes}

Some simple instances of local consistent random complex models have been described in the literature.
We go through the dense variants of these models, specifying both the corresponding complexon and the random complex model.

\citet{Linial2006,Meshulam2009} considered the homological connectivity of random simplicial complexes generated in the following way.
For some real number $p\in[0,1]$ and integers $d>1,n\geq 1$, a random simplicial complex $K_{n,p}$ is formed on the node set $[n]$ by including all $d-1$ simplices on $[n]$, and including each element of $\binom{[n]}{d+1}$ independently with probability $p$.
This model is known as the ``Linial-Meshulam model.''
One can define a complexon whose samples are distributed in this way.
Define $W\in\kerns$ so that for all $x\in[0,1]^{c+1}$ for $c<d$, we have $W(x)=1$, and for all $x\in[0,1]^{d+1}$ we have $W(x)=p$.
It is clear to see that for all $n$, the random simplicial complex $\sample{K}(n,W)$ follows the Linial-Meshulam model.

If the Linial-Meshulam assumes a fully-connected $(d-1)$-skeleton and fills in $d$-simplices at random, the random flag, or clique, complex~\citep{Kahle2009} does the opposite.
The flag complex, again for some $p\in[0,1]$, yields a random complex on nodes $[n]$ by beginning with an Erd\H{o}s-R\'{e}nyi random graph $G_n$ with edge probability $p$, and then taking the maximal simplicial complex whose $1$-skeleton is equal to $G_n$.
This can also be expressed with an appropriately constructed complexon.
Define $W\in\kerns$ so that for all $x\in[0,1]^2$, $W(x)=p$, and otherwise $W(x)=1$.
Then, the random simplicial complex $\sample{K}(n,W)$ is a random flag complex.

As a generalization of both of these, \citet{Costa2016} considers random simplicial complexes parameterized by a sequence of coefficients $p_1,p_2,p_3,\ldots$ contained in $[0,1]$.
The multiparameter, or Costa-Farber, random simplicial complex is distributed according to $\sample{K}(n,W)$, where $W$ is such that for all $x\in[0,1]^{d+1}$, $W(x)=p_d$.
Further work showed that in the dense, or medial, regime of $p_j>0$ for all $j$, the $d^{\mathrm{th}}$ homology group of such a random complex in a given dimension is nontrivial for $n$ in a logarithmically-sized interval, being trivial elsewhere with high probability~\citep{Farber2020}.
That is to say, as $n\to\infty$, the homology group of a fixed dimension for complexes sampled according to $\sample{K}(n,W)$ vanishes with probability tending to one.
For some complexons, particularly those bounded away from zero in correspondence with the medial regime studied by \citet{Farber2020}, it is seemingly difficult then to speak of the ``homological structure of a complexon'' in general.

These three models are all instances of homogenous random simplicial complexes, where the probability of a simplex existing is dependent only on its dimension.
For a complexon $W$, this corresponds to $W$ being constant on $[0,1]^{d+1}$ for each $d\geq 1$.

%% file: sections/52_5_other-random.tex
\subsection{Other Random Simplicial Complex Models}

As described in \cref{sec:complexons:sampling}, a random simplicial complex sampled from a complexon following $\sample{K}(n,W)$ essentially picks nodes $[n]$ with ``latent positions'' $x_1,\ldots,x_n$ uniformly distributed in the unit interval.
That is to say, the random complex structure is dependent on each node's position in the space $[0,1]$.
In the ensuing construction, no topological properties of $[0,1]$ were leveraged, suggesting that the definition of a complexon can be extended to nodes sampled from any Borel probability space.
Let $(\Omega,\mathcal{F},\mu)$ be a Borel probability space, which we denote by $\Omega$ for short.
A complexon on $\Omega$ is a measurable function
\begin{equation}
  W:\djunion_{d\geq 1}\Omega^{d+1}\to[0,1].
\end{equation}
All of the basic constructions so far generalize immediately to complexons on $\Omega$: sampling, homomorphism densities, cut-metrics, and so on, with the same results holding.

Defining complexons over alternative probability spaces is useful for describing random simplicial complexes tied to a natural underlying space.
We consider here dense random geometric \v{C}ech complexes~\citep{Bobrowski2022}.
Let $\mu$ be a probability measure on $\mathbb{R}^d$ for some integer $d\geq 1$, and let $\mathcal{X}$ denote the support of $\mu$.
For some $\epsilon>0$, define a complexon $W_\epsilon$ on $\mathcal{X}$ so that, for any points $x_1,\ldots,x_{d+1}\in\mathcal{X}$,
\begin{equation*}
  W_\epsilon(x_1,\ldots,x_{d+1}) = \indic(\bigcap_{j=1}^{d+1}B_{\epsilon/2}(x_j)\neq\emptyset),
\end{equation*}
where $B_{\epsilon/2}(x_j)$ indicates the open ball in $\mathbb{R}^d$ of radius $\epsilon/2$ centered about $x_j$.
That is to say, $W_\epsilon(x_1,\ldots,x_{d+1})=1$ if and only if the collection of points $\{x_1,\ldots,x_{d+1}\}$ has diameter less than $\epsilon$.

Going back to our definition of the random simplicial complex $\sample{K}(n,W_\epsilon)$, let $\mathcal{X}_n=\{x_1,\ldots,x_n\}$ be such that $x_1,\ldots,x_n$ are \iid{} samples from the probability distribution $\mu$.
Form a simplicial complex $K$ such that $\Vert(K)=[n]$.
Inductively for $d\geq 1$ until termination, for each $d+1$-subset $\sigma\subseteq[n]$ such that every strict subset of $\sigma$ is contained in $K$, include $\sigma$ in $K$ with probability $W_\epsilon(x_\sigma)$.
The random simplicial complex $K$ formed in this way is, much like before, said to be sampled from the distribution $\sample{K}(n,W_\epsilon)$.
Moreover, under our particular definition of $W_\epsilon$, it is clear that the resulting simplicial complex $K$ is the \v{C}ech complex $\check{C}_\epsilon(\mathcal{X}_n,\mathbb{R}^d)$.
In this way, $\sample{K}(n,W_\epsilon)$ is the distribution of \v{C}ech complexes with fixed radius $\epsilon$ formed by sampling $n$ \iid{} points in $\mathcal{X}$ from the probability distribution $\mu$.

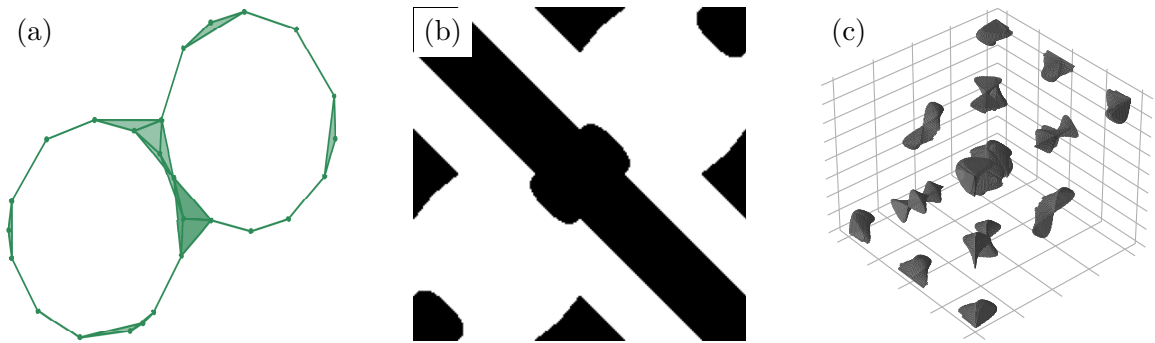
\begin{figure*}
  \centering
  \input{figs/cech.tikz.tex}
  \caption{%
    \textbf{(a)} \v{C}ech complex formed from points uniformly sampled from a bouquet of two circles in $\mathbb{R}^2$.
    \textbf{(b)} The complexon $W_\epsilon$ as defined on $[0,1]^2$ (dark regions indicate a value of $1$).
    \textbf{(c)} The complexon $W_\epsilon$ as defined on $[0,1]^3$ (dark regions indicate a value of $1$).
  }
  \label{fig:cech}
\end{figure*}

Revisiting our example of the \v{C}ech complex formed from the bouquet of two circles in \cref{sec:intro} (\cf{}~\cref{fig:bouquet-example}), we can now couch our understanding of convergent subsampled \v{C}ech complexes in the language of complexons.
As before, let $\mathcal{X}$ be the bouquet of two unit circles in $\mathbb{R}^2$.
For any $n\geq 1$, let $\mathcal{X}_n$ be a collection of $n$ \iid{} samples from the uniform distribution on $\mathcal{X}$.
An instance of the random \v{C}ech complex $\check{C}_\epsilon(\mathcal{X}_n,\mathbb{R}^2)$ is pictured in \cref{fig:cech}~(a).

We can equivalently describe the random \v{C}ech complex as a complexon.
For the purpose of illustration, parameterize the space $\mathcal{X}$ as a curve $c:[0,1]\to\mathbb{R}^2$, defined as
\begin{equation*}
  c(t) =
  \begin{cases}
    (\cos(4\pi t)-1,\sin(4\pi t)) & t<1/2 \\
    (1-\cos(4\pi t),\sin(4\pi t)) & t\geq 1/2.
  \end{cases}
\end{equation*}
The uniform measure $\mu$ on $\mathcal{X}$ is merely the pushforward measure of the uniform measure on $[0,1]$ via the map $c$.
Moreover, $c$ is injective almost everywhere, so we can represent the complexon modeling $\check{C}_\epsilon(\mathcal{X},\mathbb{R}^2)$ as a standard complexon on $[0,1]$, up to a set of measure zero.
The complexon $W_\epsilon$ as defined on $[0,1]^2$ and $[0,1]^3$ is pictured in \cref{fig:cech}~(b,c).
Indeed, $W_\epsilon$ can be thought of as the indicator function for simplices in the \v{C}ech complex $\check{C}_\epsilon(\mathcal{X},\mathbb{R}^2)$.

%% file: figs/cech.tikz.tex
\begin{tikzpicture}
  \pgfmathsetseed{1283}

  \begin{groupplot}[
    group style={
      group size=3 by 1,
      group name=myplots,
      horizontal sep=1cm,
    },
    width=6cm,
    height=6cm,
    axis lines=none,
    ylabel near ticks,
    legend cell align=left,
    ]

    \nextgroupplot[
    xmin=0,xmax=1,ymin=0,ymax=1,
    ticks=none,
    ]
    \addplot graphics[xmin=0,xmax=1,ymin=0,ymax=1] {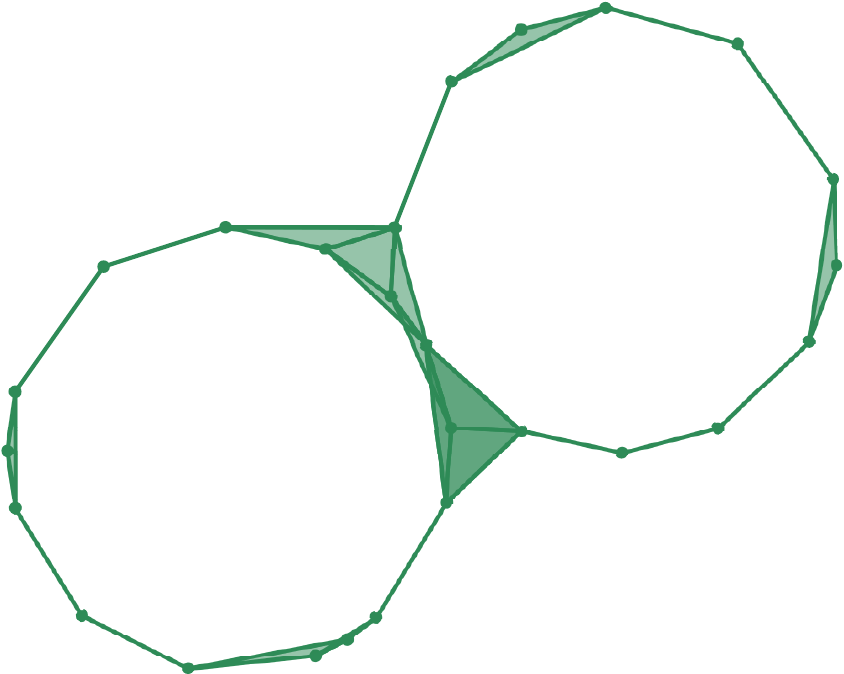};
    
    \nextgroupplot[
    xmin=0,xmax=1,ymin=0,ymax=1,
    ticks=none,
    ]
    \addplot graphics[xmin=0,xmax=1,ymin=0,ymax=1] {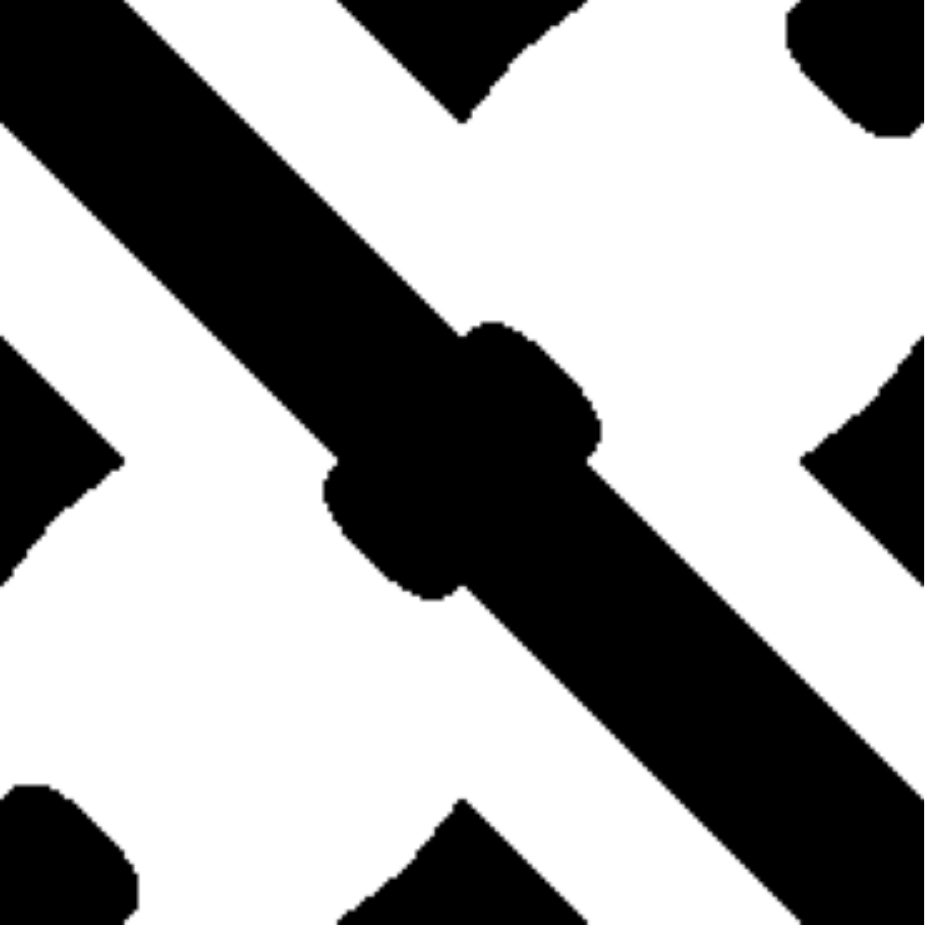};

    \nextgroupplot[
    xmin=0,xmax=1,ymin=0,ymax=1,
    ticks=none,
    ]
    \addplot graphics[xmin=0,xmax=1,ymin=0,ymax=1] {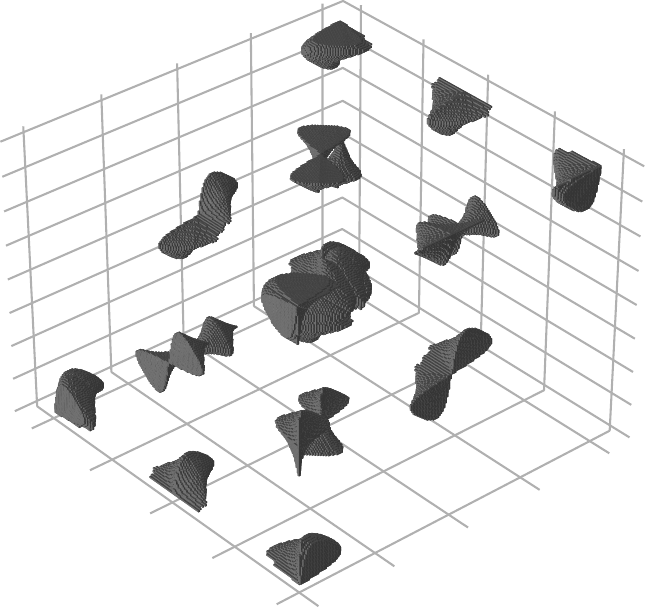};

  \end{groupplot}
  
  \foreach \plt/\lab in {c1r1/a,c2r1/b,c3r1/c} {
    \node[anchor=north west,fill=white,fill opacity=0.25,opacity=1] at (myplots \plt.north west) {(\lab)};
  }
\end{tikzpicture}

%% file: sections/53_hypergraphs.tex
\subsection{Convergence of Nonuniform Hypergraphs}
\label{sec:remarks:hypergraphs}

An open problem in the theory of graph homomorphisms and graph limits is the characterization of nonuniform hypergraphs with unbounded edges~\citep{Lovasz2008}.
As noted in \cref{sec:complexons:sampling}, the distribution $\sample{K}(n,W)$ can be thought of in terms of ``lower sampling'' following \citep{Farber2022}, where a random hypergraph $H$ is drawn according to $W$, and then taking the largest simplicial complex contained in that hypergraph.
On the other hand, the ``upper sampling'' approach~\citep{Farber2022} can be reproduced by drawing a random hypergraph $H$ according to $\faceted{W}$, and then taking the smallest simplicial complex that contains that hypergraph.
Both of these approaches yield the same distribution $\sample{K}(n,W)$.

The viewpoint of simplicial complexes formed from the upper or lower bounds of hypergraphs sheds some light on limits of dense, nonuniform hypergraphs.
For a hypergraph $H$, let $\hlower{H}$ be the largest simplicial complex contained in $H$, and $\hupper{H}$ the smallest simplicial complex containing $H$.
We say that a sequence of hypergraphs $(H_n)_{n\geq 1}$ is upper-lower-convergent (or, for brevity, UL-convergent) if the following conditions hold simultaneously for some complexon $W$:
\begin{enumerate}[label={(\roman*)}]
\item $(\hupper{H_n})_{n\geq 1}$ is convergent with limiting complexon $W$
\item $(\hlower{H_n})_{n\geq 1}$ is convergent with limiting complexon $\faceted{W}$.
\end{enumerate}
If this is the case, we say that $H_n\to W$ as $n\to\infty$.
Since for all $n$, $\hupper{H_n}\in\faceted{\kerns}$, by~\cref{lemma:cut-compactness,lemma:facet-continuous}, it holds that $W\in\faceted{\kerns}$.

Based on \cref{lemma:sampling}, we know that large simplicial complexes sampled from a complexon are close to the faceted complexon in the cut-metric, with high probability.
Based on the relationship between faceted complexons and upper/lower sampling, we establish a similar approach to show that sequences of hypergraphs sampled from faceted complexons are UL-convergent.
Let $W$ be a complexon, and $n\geq 1$ an integer.
Similar to the definition of $K\sim\sample{K}(n,W)$, we define a random hypergraph $H$ as follows.
First, take $n$ \iid{} samples $x_1,\ldots,x_n$ uniformly at random in the unit interval $[0,1]$.
Let $H$ be a hypergraph with $\Vert(H)=[n]$.
Then, for each $\sigma\in 2^{[n]}$, include $\sigma\in H$ with probability $W(x_\sigma)$.
We call the distribution from which $H$ is drawn $\sample{H}(n,W)$.

\begin{lemma}\label{lemma:hypergraph-sampling}
  For a complexon $W$ and integer $n\geq 1$, it holds that
  \begin{equation*}
    \sample{K}(n,W)
    \overset{D}{=} \hlower{\sample{H}(n,W)}
    \overset{D}{=} \hupper{\sample{H}(n,\faceted{W})},
  \end{equation*}
  where $\overset{D}{=}$ denotes equality in distribution.
\end{lemma}

We now consider the relationship between UL-convergence, simplicial homomorphisms, and hypergraph homomorphisms.
For a hypergraph $H$, define a complexon $W_H$ in the same way as one does for simplicial complexes, \ie, by taking indicator functions of tuples defined on a partition of $[0,1]$.
From this, we inherit all the usual notions of homomorphisms and cut-metrics for hypergraphs.
For a hypergraph $H$ and a complexon $W$, define
\begin{equation*}
  \thom(H,W) = \int_{[0,1]^{\Vert(H)}}\prod_{\sigma\in H}W(x_\sigma)\prod_{j\in\Vert(H)}\drv{x_j}.
\end{equation*}

With the notion of a homomorphism for hypergraphs, it is natural to ask whether the convergence of a sequence of hypergraphs in all hypergraph homomorphisms implies UL-convergence.
Indeed, convergence of all hypergraph homomorphism densities is enough to establish convergence of the sequence $(\hlower{H_n})$, but not to establish convergence of $(\hupper{H_n})$.

\begin{counterexample}\label{counter:uniform-hypergraphs}
  For some $p\in[0,1]$ define $W_p\in\kerns$ so that $W(x)=p$ for $x\in[0,1]^3$, and $W(x)=0$ otherwise.
  Define a sequence of hypergraphs $(H_{n,p})_{n\geq 1}$ where each $H_{n,p}$ is a random hypergraph distributed according to $\sample{H}(n,W_p)$.
  One can check that with probability $1$, for any hypergraph $G$, $t(G,H_{n,p})$ is convergent, so that $(H_{n,p})$ is a convergent sequence in the sense of homomorphisms.

  Setting $p=1$, it is obvious that $(H_{n,1})$ is not UL-convergent, since $\hlower{H_{n,1}}$ is the empty simplicial complex on $n$ nodes, and $\hupper{H_{n,1}}$ is the complete $2$-dimensional simplicial complex on $n$ nodes.
  Then, one can see that $\hupper{H_{n,1}}\to W$, where $W(x)=1$ for $x\in[0,1]^{d+1}$ when $d\in\{1,2\}$, and $W(x)=0$ otherwise.
  On the other hand, $\hlower{H_{n,1}}\to 0$.
  It is not true that $\faceted{W}=0$ in this case, so UL-convergence does not hold.
\end{counterexample}

In the above counterexample, the hypergraphs $(H_{n,1})$ were drawn from a complexon $W_1\in\kerns\setminus\faceted{\kerns}$, \ie, ``non-faceted'' complexons.
In light of the UL-limit of a convergent hypergraph sequence necessarily being an element of $\faceted{\kerns}$, this may imply some utility of the extra structure imposed by $\faceted{\kerns}$.
This is noted in the following proposition.

\begin{proposition}\label{prop:faceted-hypergraphs}
  Let $W\in\faceted{\kerns}$ be given, and take a sequence $(H_n)_{n\geq 1}$ of random hypergraphs $H_n\sim\sample{H}(n,W)$.
  Then, $(H_n)$ is UL-convergent with limit $W$ with probability $1$.
\end{proposition}

\begin{proof}
  For the sequence $(H_n)$, note that for each $n$, $\hlower{H_n}$ is distributed according to $\sample{K}(n,W)$ by \cref{lemma:hypergraph-sampling}.
  By \cref{coro:sampling-convergence}, this implies that $(\hlower{H_n})\to\faceted{W}$ with probability $1$.

  Similarly, since $W\in\faceted{\kerns}$, there is some $U\in\kerns$ such that $W=\faceted{U}$.
  Then, for each $n$, $\hupper{H_n}$ is distributed according to $\sample{K}(n,U)$, again by \cref{lemma:hypergraph-sampling}.
  Thus, $(\hupper{H_n})\to W$ with probability $1$.
\end{proof}

\Cref{prop:faceted-hypergraphs} and \cref{counter:uniform-hypergraphs} together indicate that considering the upper and lower simplicial complexes of a sequence of hypergraphs may help establish a theory of nonuniform hypergraphs, but fails in cases where the limiting object is not a faceted complexon.

%% file: sections/60_proofs.tex

\section{Proof of \Cref{lemma:sampling}}
\label{sec:proofs:sampling}

\subsection{Equipartitioning}

Let $\partition{P}$ be a partition of $[0,1]$ into finitely many measurable sets.
For a complexon $W$, define the projection $W_{\partition{P}}$ as follows.
For $\{p_j\in P_j:j\in[d+1]\}$, where each $P_j$ is an element of $\partition{P}$, put
\begin{equation*}
  W_{\partition{P}}(p_1,\ldots,p_{d+1}) = \int_{\prod_{j=1}^{d+1}P_j} W(x_1,\ldots,x_{d+1})\drv{x}.
\end{equation*}
The complexon $W_{\partition{P}}$ is an instance of a stepfunction on $\partition{P}$, \ie, $W_{\partition{P}}(x_1,\ldots,x_{d+1})$ is only dependent on which sets in $\partition{P}$ the points $x_1,\ldots,x_{d+1}$ are contained in.
We establish a result stating that any complexon can be approximated in the cut-distance by a piecewise constant complexon with respect to an equipartition of the unit interval.
The following is a weak regularity lemma, stating that any complexon can be approximated by a stepfunction over an appropriate partition.
It is a trivial extension of \citep[Theorem~12]{Frieze1999}.

\begin{theorem}\label{thm:partition}
  Let $W$ be a complexon.
  For given integers $n>1,d\geq 1$, there is a partition $\partition{P}$ of $[0,1]$ into $n$ measurable sets such that there exists a stepfunction $U$ on $\partition{P}$ such that for all $1\leq j\leq d$,
  \begin{equation*}
    \labelcutd_{\ncut{j}}(W,U)\leq\sqrt{\frac{d+1}{\log_2{n}}}.
  \end{equation*}
\end{theorem}

Following \citep{Lovasz2012}, we strengthen \cref{thm:partition} to hold for equipartitions, where an equipartition of a measure space is a partition such that each set has equal measure.

\begin{lemma}[Equipartitioning]\label{lemma:equipartitioning}
  Let $W$ be a totally symmetric, measurable function.
  For some $n>1,d\geq 1$, put $m=\lceil n^{1/4}\rceil$.
  Then, there exists an $m$-equipartition $\partition{P}$ of $[0,1]$ such that for all finite, nonnegative sequences of real numbers $\alpha_1,\ldots,\alpha_d$,
  \begin{equation*}
    \labelcutd_{\cut}(W,W_{\partition{P}};(\alpha_j))\leq
    \frac{\sum_{j=1}^d\alpha_j\left(4\sqrt{d+1}+2^j\right)}{\sqrt{\log_2{n}}}.
  \end{equation*}
\end{lemma}

\begin{proof} 
  Let $\partition{Q}=\{Q_j\}_{j=1}^m$ be an $m$-partition as guaranteed by \cref{thm:partition}, with corresponding stepfunction $U$.
  Partition each class $Q_j$ into sets of measure $1/n$, with at most one exceptional set of measure less than $1/n$ for each class.
  Notice that the measure of the union of the exceptional sets is bounded by $m/n$.
  Take the union of the exceptional sets, and partition it into sets of measure $1/n$, yielding an equipartition $\partition{P}$.

  Take the common refinement $\partition{R}=\partition{P}\wedge\partition{Q}$.
  Note that $\partition{P}$ and $\partition{R}$ only differ over those classes in $\partition{P}$ formed from exceptional sets.
  Thus, for all $j\geq 1$, the restrictions of $W_{\partition{P}}$ and $W_{\partition{R}}$ to $[0,1]^{j+1}$ only differ on a set of measure at most $2^j(m/n)$.
  This implies that
  \begin{equation*}
    \labelcutd_{\ncut{j}}(W_{\partition{R}},W_{\partition{P}})\leq
    2^j\frac{m}{n}.
  \end{equation*}
  Via an easily shown variation of \citep[Lemma~9.12]{Lovasz2012}, we see that
  \begin{equation*}
    \labelcutd_{\ncut{j}}(W,W_{\partition{R}})\leq
    2\labelcutd_{\ncut{j}}(W,U),
  \end{equation*}
  which implies via the triangle inequality that
  \begin{equation*}
    \labelcutd_{\ncut{j}}(W,W_{\partition{P}})\leq
    2\labelcutd_{\ncut{j}}(W,U)+2^j\frac{m}{n}.
  \end{equation*}
  One can then show via straightforward calculation that
  \begin{equation*}
    \labelcutd_{\ncut{j}}(W,W_{\partition{P}})\leq
    \frac{4\sqrt{d+1}+2^j}{\sqrt{\log_2{n}}}.
  \end{equation*}
  This yields the bound
  \begin{equation*}
    \labelcutd(W,W_{\partition{P}};(\alpha_j))=
    \sum_{j=1}^d\alpha_j\labelcutd_{\ncut{j}}(W,W_{\partition{P}})\leq
    \frac{\sum_{j=1}^d\alpha_j\left(4\sqrt{d+1}+2^j\right)}{\log_2{n}},
  \end{equation*}
  as desired.
\end{proof}

\subsection{Sample Concentration}

We show that ``nice'' parameters of simplicial complexes sampled from a complexon concentrate around their expectation.

\begin{definition}[Reasonably Smooth Parameter]\label{defn:smooth}
  Let $f$ be a function mapping simplicial complexes to the real numbers.
  Suppose that for any simplicial complexes $F,G$ such that $\Vert(F)=\Vert(G)$ and whose structure varies only on faces incident to a single node, \ie, there exists some $v\in\Vert(F)$ such that for all $\sigma\in F$,
  \begin{equation*}
    \sigma\in F\triangle G \Rightarrow v\in\sigma,
  \end{equation*}
  we have
  \begin{equation*}
    |f(F)-f(G)|\leq 1.
  \end{equation*}
  Under these conditions, we say that $f$ is a reasonably smooth parameter.%
  \footnote{$F\triangle G$ denotes the symmetric difference between sets.}
\end{definition}

We will need the following result \citep[Corollary~A.15]{Lovasz2012}, which is a corollary of Azuma's inequality.
\begin{proposition}\label{prop:iid-concentration}
  Let $(\Omega,\mathcal{A},\pi)$ be a probability space.
  Let $f:\Omega^n\to\reals$ be a measurable function such that
  \begin{equation*}
    |f(x)-f(y)|\leq 1
  \end{equation*}
  whenever $x$ and $y$ differ in one coordinate only.
  Let $X$ be a random point in $\Omega^n$ distributed according to the product measure.
  Then, for any $\epsilon>0$,
  \begin{equation*}
    \Pr\left\{f(X)\geq\Ex[f(X)]+\epsilon n\right\}
    \leq \exp\left(-\epsilon^2 n/2\right).
  \end{equation*}
\end{proposition}

\Cref{prop:iid-concentration} gives us a generic, one-sided tail bound for well-behaved functions on sequences of $\iid$ random variables.
We now construct the random simplicial complex $\sample{K}(n,W)$ in a way that fits this schema.

Let $W$ be a complexon, and let $n\geq 1$ be given.
Let $\Omega=\prod_{d=0}^{n-1}[0,1]^{([n]^d)}$ with the uniform probability measure.
Take \iid{} samples $x_j=(\alpha_{j,0},\ldots,\alpha_{j,n-1})$ from $\Omega$ for $j\in[n]$.
We construct a simplicial complex $F$ from the points $\{x_j\}_{j=1}^n$ in the following way.

For all $\sigma\subseteq[n]$, put $\sigma\in F$ if and only if the following condition holds.
For any $\sigma'\subseteq\sigma$, denote the ordered elements of $\sigma'$ by $1\leq i_1\leq\ldots\leq i_{d+1}\leq n$.
Then, we require for all such $\sigma'$
\begin{equation*}
  \alpha_{i_1,d}(i_2,\ldots,i_{d+1}) \leq
  W(\alpha_{i_1,0},\ldots,\alpha_{i_{d+1},0}).
\end{equation*}

Under this model, a random simplicial complex $F$ distributed according to $\sample{K}(n,W)$ is taken as a function of the \iid{} sequence $x_j$ and the complexon $W$.
Moreover, changing any coordinate $x_j$ in the sequence only affects simplices that contain $j$.
We can now state the main result of this section.

\begin{theorem}[Sample Concentration]\label{thm:sample-concentration}
  Let $f$ be a reasonably smooth parameter.
  Then, for a complexon $W$, an integer $n\geq 1$, and real number $t\geq 0$, we have
  \begin{equation*}
    \Pr\left\{f(\sample{K}(n,W))\geq\Ex[f(\sample{K}(n,W))]+\sqrt{2tn}\right\}\leq
    \exp\left(-t\right).
  \end{equation*}
\end{theorem}

\begin{proof}
  Follows from \cref{prop:iid-concentration}, treating $f(\sample{K}(n,W))$ as the composition of functions of the sequence of \iid{} random variables $\{x_j\in\Omega\}_{j=1}^n$.
\end{proof}

\subsection{Concentration of Norm}

The following lemma will be the most challenging part of proving \cref{lemma:sampling}.

\begin{lemma}[Norm Concentration]\label{lemma:norm-concentrate}
  Let $U:[0,1]^{d+1}\to[-1,1]$, and let $X$ be a random ordered $n$-subset of $[0,1]$.
  Let $U[X]$ be the matrix in $[-1,1]^{n^{d+1}}$ determined by evaluating $U$ on $X^{d+1}$.
  Then with probability at least $1-4\exp(-n^{1/(d+1)}/2)$,
  \begin{equation*}
    -3r(n,d+1) \leq \|U[X]\|_{\ncut{d}} - \|U\|_{\ncut{d}} \leq \frac{5(d+1)^2}{\sqrt{n^{1/(d+1)}}},
  \end{equation*}
  where
  \begin{equation*}
    r(n,d+1) = 1-\frac{P(n,d+1)}{n^{d+1}}.
  \end{equation*}
\end{lemma}

We begin with the following result, which is a special case of \citep[Lemma~3]{Alon2003}.

\begin{lemma}\label{lemma:random-cut-bounded}
  Let $B:[n]^{d+1}\to[-1,1]$.
  Let $S_1,\ldots,S_{d+1}\subseteq [n]$, and for each $j\in[d+1]$, let $Q_j$ be a random subset of $[n]^{[d+1]\setminus\{j\}}$ with cardinality $q$.
  Then,
  \begin{equation*}
    B(S_1,\ldots,S_{d+1})\leq
    \Ex_{Q_j}\left[B\left(S_1,\ldots,S_{j-1},(Q_j\cap \prod_{i\neq j}S_i)^+,S_{j+1},\ldots,S_{d+1}\right)\right]+
    \frac{n^{d+1}}{\sqrt{q}},
  \end{equation*}
  where
  \begin{equation*}
    B(S_1,\ldots,S_{d+1}) = \sum_{v_1\in S_1}\cdots\sum_{v_{d+1}\in S_{d+1}}B(v_1,\ldots,v_{d+1}),
  \end{equation*}
  and $(Q_j\cap \prod_{i\neq j}S_i)^+$ is the collection of elements $v\in [n]$ such that,
  \begin{equation*}
    \sum_{\vect{q}\in Q_j\cap\prod_{i\neq j}S_i}B(q_1,\ldots,q_{j-1},v,q_{j+1},\ldots,q_{d+1}) > 0.
  \end{equation*}
\end{lemma}

\Cref{lemma:random-cut-bounded} can be iterated over all dimensions of the matrix $B$ in order to yield an expectational bound on a variation of the cut-norm for a matrix.
We find it useful to consider the one-sided cut-norm, defined as follows.
\begin{definition}
  Let $B:[n]^{d+1}\to\reals$ and $U:[0,1]^{d+1}\to\reals$.
  Define
  \begin{align*}
    \|B\|_{\ncut{d}}^+ &= \frac{\max_{S_1,\ldots,S_{d+1}\subseteq[n]}B(S_1,\ldots,S_{d+1})}{n^{d+1}} \\
    \|U\|_{\ncut{d}}^+ &= \sup_{S_1,\ldots,S_{d+1}\in\Borel[0,1]}\int_{\prod_j S_j}U(x)\drv{x}.
    \end{align*}
\end{definition}

Observing that $\|B\|_{\ncut{d}}=\max\{\|B\|_{\ncut{d}}^+;\|-B\|_{\ncut{d}}^+\}$, our goal is to find bounds for $\|U[X]\|_{\ncut{d}}^+-\|U\|_{\ncut{d}}^+$ and apply them twice, yielding the desired result.

\begin{lemma}\label{lemma:random-cut-multi}
  Let $B:[n]^{d+1}\to[-1,1]$.
  Let $\{Q_j\}_{j=1}^{d+1}$ each be random subsets of $[n]^{[d+1]\setminus\{j\}}$ with cardinality $q$.
  Then,
  \begin{equation*}
    \|B\|_{\ncut{d}}^+ \leq
    \frac{1}{n^{d+1}}\Ex_{Q_1,\ldots,Q_{d+1}}\left[\max_{R_j\subseteq Q_j}B(R_1^+,\ldots,R_{d+1}^+)\right]+
    \frac{d+1}{\sqrt{q}}.
  \end{equation*}
\end{lemma}

\begin{proof}
  Fix $S_1,\ldots,S_{d+1}\in[n]$.
  Applying \cref{lemma:random-cut-bounded} repeatedly over each coordinate $j\in[d+1]$ yields
  \begin{equation*}
    B(S_1,\ldots,S_{d+1}) \leq
    \Ex_{Q_1,\ldots,Q_{d+1}}\left[\max_{R_j\subseteq Q_j}B(R_1^+,\ldots,R_{d+1}^+)\right]+
    \frac{(d+1)n^{d+1}}{\sqrt{q}}.
  \end{equation*}
  Taking the supremum of the left-hand side over all possible sets $S_1,\ldots,S_{d+1}$ and normalizing by $n^{d+1}$ yields the desired bound.
\end{proof}

The following lemma immediately precedes our desired result.

\begin{lemma}\label{lemma:pre-first-sampling}
  Let $U:[0,1]^{d+1}\to[-1,1]$, and let $X$ be a random ordered $n$-subset of $[0,1]$.
  Then with probability at least $1-2\exp(-n^{1/(d+1)}/2)$,
  \begin{equation*}
    -3r(n,d+1) \leq \|U[X]\|_{\ncut{d}}^+ - \|U\|_{\ncut{d}}^+ \leq \frac{5(d+1)^2}{\sqrt{n^{1/(d+1)}}}.
  \end{equation*}
\end{lemma}

\begin{proof}
  For a random $n$-subset $X\subseteq[0,1]$, let $B=U[X]$.

  \medskip
  \noindent\textit{(Lower bound)}
  For any measurable subsets $S_1,\ldots,S_{d+1}\subseteq[0,1]$ we have
  \begin{equation*}
    \|B\|_{\ncut{d}}^+\geq\frac{1}{n^{d+1}}U(S_1\cap X, \ldots, S_{d+1}\cap X),
  \end{equation*}
  where $U(S_1\cap X,\ldots,S_{d+1}\cap X)$ is used as shorthand for
  \begin{equation*}
    U(S_1\cap X, \ldots, S_{d+1}\cap X) =
    \sum_{x_1\in S_1\cap X} \cdots \sum_{x_{d+1}\in S_{d+1}\cap X}U(x_1,\ldots,x_{d+1}).
  \end{equation*}
  On average, we see that
  \begin{align*}
    \Ex_X\left[\|B\|_{\ncut{d}}^+\right] &\geq \frac{1}{n^{d+1}}\Ex_X\left[U(S_1\cap X,\ldots,S_{d+1}\cap X)\right] \\
                                               &= \frac{P(n,d+1)}{n^{d+1}}\int_{\prod_j S_j}U(x)\drv{x} + r(n,d+1)\cdot g(U;S_1,\ldots,S_{d+1}),
  \end{align*}
  where $g$ is some function taking values with magnitude at most $1$.
  Thus, since the sets $S_1,\ldots,S_{d+1}$ were given arbitrarily, we can choose them to yield the following inequality.
  \begin{align*}
    \Ex_X\left[\|B\|_{\ncut{d}}^+\right]&\geq\int_{\prod_j S_j}U(x)\drv{x} - 2r(n,d+1) \\
                                              &\geq \|U\|_{\ncut{d}}^+ - 2r(n,d+1).
  \end{align*}
  Applying \cref{prop:iid-concentration} yields
  \begin{equation*}
    \|B\|_{\ncut{d}}^+-\|U\|_{\ncut{d}}^+\geq -3r(n,d+1)
  \end{equation*}
  with probability at least $1-\exp(-r^2(n,d+1)/2)$.

  \medskip
  \noindent\textit{(Upper bound)}
  Let $Q_1,\ldots,Q_{d+1}$ be, for each $j$, random $q$-subsets of $[n]^{[d+1]\setminus\{j\}}$.
  We will bound the expectation of $\|B\|_{\ncut{d}}^+$ by first fixing some parameters.
  Fix the sets $R_j\subseteq Q_j\subseteq[n]^d$, and define $Q=\bigcup_j \{Q_{j,1},\ldots,Q_{j,d}\}$, so that $Q\subseteq [n]$.
  Fix those values of $X_i$ for which $i\in Q$.
  Define for each $j\in[d+1]$ the set
  \begin{equation*}
    Y_j = \left\{y\in[0,1]:\sum_{(i_1,\ldots,i_d)\in R_j}U(X_{i_1},\ldots,X_{i_{j-1}},y,X_{i_j},\ldots,X_{i_d})>0\right\}.
  \end{equation*}
  Let $X'=\{X_i:i\in[n]\setminus Q\}$.
  For every collection of indices $\{i_j\in[n]\setminus Q\}_{j=1}^{d+1}$, the contribution of $U(X_{i_1},\ldots,X_{i_{d+1}})$ to $\Ex_{X'}[B(R_1^+,\ldots,R_{d+1}^+)]$ is at most
  \begin{equation*}
    \int_{\prod_j Y_j}U(y_1,\ldots,y_{d+1})\drv{y}\leq\|U\|_{\ncut{d}}^+.
  \end{equation*}
  There are fewer than $n^{d+1}$ such choices of indices.
  The remaining terms where $i_j\in Q$ for at least one $j\in[d+1]$ contribute at most $(2n|Q|)^d\leq (2(d+1)nq)^d$ in absolute value, yielding the bound
  \begin{equation*}
    \Ex_{X'}\left[B(R_1^+,\ldots,R_{d+1}^+)\right]\leq n^{d+1}\|U\|_{\ncut{d}}^++2^d\left((d+1)nq\right)^d.
  \end{equation*}
  Note that $B(R_1^+,\ldots,R_{d+1}^+)$ is a function of the random variables $X'$.
  Moreover, $B(R_1^+,\ldots,R_{d+1}^+)$ varies by at most $4n^d$ when a single coordinate of $X'$ is changed.
  Thus, by \cref{prop:iid-concentration}, we have for a given $\epsilon>0$, with probability at least $1-\exp(-\epsilon^2 n/2)$,
  \begin{equation*}
    B(R_1^+,\ldots,R_{d+1}^+)\leq\Ex_{X'}\left[B(R_1^+,\ldots,R_{d+1}^+)\right]+4\epsilon n^{d+1}.
  \end{equation*}
  The number of possible choices of sets $R_1,\ldots,R_{d+1}$ is at most $2^{q(d+1)}$, so that the above bound holds for all $\{R_j\subset Q_j\}_{j=1}^{d+1}$ with probability at least $1-2^{q(d+1)}\exp(-\epsilon^2 n/2)$.
  This property is preserved when taking the expectation over $Q_1,\ldots,Q_{d+1}$, so that, by \cref{lemma:random-cut-multi},
  \begin{align*}
    \|B\|_{\ncut{d}}^+ &\leq \frac{1}{n^{d+1}}\Ex_{Q_1,\ldots,Q_{d+1}}\left[\max_{R_j\subseteq Q_j}B(R_1^+,\ldots,R_{d+1}^+)\right] + \frac{d+1}{\sqrt{q}} \\
                      &\leq \|U\|_{\ncut{d}}^+ + \frac{(2(d+1)q)^d}{n} + 4\epsilon + \frac{d+1}{\sqrt{q}}
  \end{align*}
  with probability at least $1-2^{q(d+1)}\exp\left(-\epsilon^2 n/2\right)$.

  Put $q=\lfloor n^{1/(d+1)}\rfloor/(2(d+1))$ and $\epsilon=\sqrt{2(d+1)/q}$, noting that $q\geq 1$ for $n$ large enough to make the desired bound nontrivial.
  This yields the bound
  \begin{equation*}
    \|B\|_{\ncut{d}}^+-\|U\|_{\ncut{d}}^+\leq\frac{5(d+1)^2}{\sqrt{n^{1/(d+1)}}}
  \end{equation*}
  with probability at least $1-\exp(-n^{1/(d+1)}/2)$.

  \medskip
  \noindent\textit{(Conclusion of Proof)}
  Combining the lower and upper bounds, we see that
  \begin{equation*}
    -3r(n,d+1)\leq\|B\|_{\ncut{d}}^+-\|U\|_{\ncut{d}}^+\leq\frac{5(d+1)^2}{\sqrt{n^{1/(d+1)}}}
  \end{equation*}
  with probability at least $1-2\exp(-n^{1/(d+1)}/2)$, as desired.
\end{proof}

\begin{proof}[Proof of \cref{lemma:norm-concentrate}]
  Apply \cref{lemma:pre-first-sampling} to both $U$ and $-U$.
  Since $\|U\|_{\ncut{d}}=\max\{\|U\|_{\ncut{d}}^+,\|-U\|_{\ncut{d}}^+\}$, this yields the desired tail bound.
\end{proof}

\subsection{Sampling Lemma}

With these results established, we now prove \cref{lemma:sampling}.
We will first need the following result.

\begin{lemma}\label{lemma:exp-weighted-sample}
  Let $(H,\omega)$ be a weighted simplicial complex on $n$ nodes.
  Let $\sample{K}(H)$ be the random simplicial complex drawn from $H$.
  For every $p\in[0,1]$, we have
  \begin{equation*}
    \Ex\left[\labelcutd_{\ncut{d}}\left(\sample{K}(H),\faceted{H}\right)\right]\leq
    \frac{3(d+1)^p+1}{n^p}.
  \end{equation*}
\end{lemma}

\begin{proof}
  Let $\epsilon>0$ be given.
  For $\{i_j\in[n]\}_{j=1}^{d+1}$, define the random variable $X_{\vect{i}}=\mathbf{1}(\vect{i}\in\sample{K}(H))$.
  Let $S_1,\ldots,S_{d+1}$ be pairwise disjoint subsets of $[n]$, and let $A_{K,d}$ and $A_{\faceted{H},d}$ be the (weighted) $d$-dimensional adjacency matrices associated to $\sample{K}(H)$ and $\faceted{H}$, respectively.
  We call such a collection of sets \emph{bad} if
  \begin{equation*}
    \left|(A_{K,d}-A_{\faceted{H},d})(S_1,\ldots,S_{d+1})\right| >
    \epsilon \left(\frac{n}{d+1}\right)^{d+1},
  \end{equation*}
  Note that this is only possible if $\left|\prod_j S_j\right|>\epsilon(n/(d+1))^{d+1}$.
  Then, by the Chernoff-Hoeffding inequality,
  \begin{equation*}
    \Pr_{\sample{K}(H)}
    \left\{\left|\sum_{\vect{i}\in\prod_j S_j}(X_{\vect{i}}-\Ex[X_{\vect{i}}])\right|
      >\epsilon\left(\frac{n}{d+1}\right)^{d+1}\right\}
    \leq
    2\exp\left(-2\epsilon\left(\frac{n}{d+1}\right)^{d+1}\right).
  \end{equation*}
  There are $(d+2)^n$ such collections of disjoint subsets, so the probability of a bad collection existing is bounded by
  \begin{equation*}
    2\cdot(d+2)^n\exp\left(-2\epsilon\left(\frac{n}{d+1}\right)^{d+1}\right).
  \end{equation*}
  By \cref{lemma:disjoint-cut,lemma:finite-to-complexon-cut}, the nonexistence of bad subsets implies $\labelcutd_{\ncut{d}}(\sample{K}(H),\faceted{H})\leq\epsilon$.
  Therefore, we have
  \begin{equation*}
    \Ex\left[\labelcutd_{\ncut{d}}(\sample{K}(H),\faceted{H})\right] \leq
    \epsilon + 2\cdot(d+2)^n\exp\left(-2\epsilon\left(\frac{n}{d+1}\right)^{d+1}\right).
  \end{equation*}
  It is sufficient to consider the case when $n^p>3(d+1)^p+1$, otherwise the bound is trivial.
  In particular, taking $\epsilon=3\left(\frac{d+1}{n}\right)^p$ yields the desired bound.
\end{proof}



\begin{proof}[Proof of \cref{lemma:sampling}]
  Let integers $d\geq 1, n>1$ be given.
  We bound the expected cut-distance for any finite, nonnegative sequence $\alpha_1,\ldots,\alpha_{d}$ as follows.
  For some partition $\partition{P}$ of $[0,1]$ and a random $n$-subset $S\subseteq[0,1]$:
  \begin{subequations}
    \begin{align}
      \Ex\left[\cutd_{\cut}(\faceted{W},\sample{K}(W[S]);(\alpha_j))\right]
      &\leq \labelcutd_{\cut}(\faceted{W},\faceted{W}_{\partition{P}};(\alpha_j)) \label{eq:partition-bound}\\
      &\quad+ \Ex\left[\cutd_{\cut}(\faceted{W}_{\partition{P}},\faceted{W}_{\partition{P}}[S];(\alpha_j))\right] \label{eq:simple-bound}\\
      &\quad+ \Ex\left[\labelcutd_{\cut}(\faceted{W}_{\partition{P}}[S],\faceted{W}[S];(\alpha_j))\right] \label{eq:first-sampling-bound}\\
      &\quad+ \Ex\left[\labelcutd_{\cut}(\faceted{W}[S],\sample{K}(W[S]);(\alpha_j))\right], \label{eq:random-complex-bound}
    \end{align}
  \end{subequations}
  by the triangle inequality.
  We now bound each of the above summands individually.

  \medskip
  \noindent\textit{Bound~\eqref{eq:partition-bound}:}
  Let $m=\lceil n^{1/4}\rceil$.
  By \cref{lemma:equipartitioning}, there is an $m$-equipartition $\partition{P}=\{P_j\}_{j=1}^m$ of $[0,1]$ such that
  \begin{equation*}
    \labelcutd_{\cut}(\faceted{W},\faceted{W}_{\partition{P}};(\alpha_j))\leq
    \frac{\sum_{j=1}^d\alpha_j\left(4\sqrt{d+1}+2^j\right)}{\sqrt{\log_2{n}}}
  \end{equation*}
  for all finite, nonnegative sequences $\alpha_1,\ldots,\alpha_d$, where $\faceted{W}_{\partition{P}}$ denotes the projection of $\faceted{W}$ to a stepfunction on $\partition{P}$ (as opposed to the faceted complexon corresponding to $W_{\partition{P}}$).

  \medskip
  \noindent\textit{Bound~\eqref{eq:simple-bound}:}
  Let $H=\faceted{W}_{\partition{P}}[S]$, and denote the corresponding complexon by $\faceted{W}_{H}$.
  Following the proof of \citep[Lemma~10.16]{Lovasz2012}, observe that both $\faceted{W}_{H}$ and $\faceted{W}_{\partition{P}}$ are stepfunctions on $m$-partitions of $[0,1]$, with the same function values on corresponding steps.
  Each step of $\faceted{W}_{\partition{P}}$ has measure $1/m$, while each step of $\faceted{W}_{H}$ has measure $|P_i\cap S|/n=1/m+r_i$ for $i\in[m]$, where $r_i$ denotes some residual.
  It is easy to see that
  \begin{equation*}
    \cutd_{\cut}(\faceted{W}_{\partition{P}},\faceted{W}_{H};(\alpha_j)) \leq \sum_{i=1}^m|r_i|\sum_{j=1}^d\alpha_j(j+1).
  \end{equation*}
  By the estimate in the proof of \citep[Lemma~10.16]{Lovasz2012}, we have
  \begin{equation*}
    \Ex\left[\cutd_{\cut}(\faceted{W}_{\partition{P}},\faceted{W}_{H};(\alpha_j))\right]\leq
    \Ex\left[\sum_{i=1}^m|r_i|\right]\sum_{j=1}^d\alpha_j(j+1)<
    \frac{\sum_{j=1}^d\alpha_j(j+1)}{n^{3/8}}.
  \end{equation*}

  \medskip
  \noindent\textit{Bound~\eqref{eq:first-sampling-bound}:}
  By \cref{lemma:norm-concentrate}, we have
  \begin{equation*}
    \left| \labelcutd_{\ncut{d}}(\faceted{W}[S],\faceted{W}_{\partition{P}}[S]) - \labelcutd_{\ncut{d}}(\faceted{W},\faceted{W}_{\partition{P}}) \right| \leq \frac{5(d+1)^2}{\sqrt{n^{1/(d+1)}}}
  \end{equation*}
  with probability at least $1-8\exp(-n^{1/(d+1)}/2)$.
  This implies that
  \begin{equation*}
    \Ex\left[\left| \labelcutd_{\ncut{d}}(\faceted{W}[S],\faceted{W}_{\partition{P}}[S]) - \labelcutd_{\ncut{d}}(\faceted{W},\faceted{W}_{\partition{P}}) \right|\right] \leq \frac{5(d+1)^2+5}{\sqrt{n^{1/(d+1)}}}.
  \end{equation*}
  Applying the triangle inequality yields
  \begin{align*}
    \Ex\left[\labelcutd_{\ncut{d}}(\faceted{W}[S],\faceted{W}_{\partition{P}}[S])\right]
    &\leq \Ex\left[\left| \labelcutd_{\ncut{d}}(\faceted{W}[S],\faceted{W}_{\partition{P}}[S]) - \labelcutd_{\ncut{d}}(\faceted{W},\faceted{W}_{\partition{P}}) \right|\right] \\
    &\qquad+ \labelcutd_{\ncut{d}}(\faceted{W},\faceted{W}_{\partition{P}}) \\
    &\leq \frac{5(d+1)^2+5}{\sqrt{n^{1/(d+1)}}} + \frac{2^d+4\sqrt{d+1}}{\sqrt{\log_2{n}}} \\
    &\leq \frac{2^{d+2}}{\sqrt{\log_2{n}}},
  \end{align*}
  where the final inequality holds by assuming that $n$ is large enough so that the bound in \cref{lemma:sampling} is nontrivial.
  By linearity of expectation, this yields
  \begin{equation*}
    \Ex\left[\labelcutd_{\cut}(\faceted{W}[S],\faceted{W}_{\partition{P}}[S];(\alpha_j))\right]\leq
    \frac{\sum_{j=1}^d\alpha_j\left(2^{j+2}\right)}{\sqrt{\log_2{n}}}.
  \end{equation*}

  \medskip
  \noindent\textit{Bound~\eqref{eq:random-complex-bound}:}
  By \cref{lemma:exp-weighted-sample}, we have
  \begin{equation*}
    \Ex\left[\labelcutd_{\ncut{d}}(\faceted{W}[S],\sample{K}(W[S]))\right]\leq
    \frac{3(d+1)^{3/8}+1}{n^{3/8}},
  \end{equation*}
  which implies by linearity of expectation
  \begin{equation*}
    \Ex\left[\labelcutd_{\cut}(\faceted{W}[S],\sample{K}(W[S]);(\alpha_j))\right]\leq
    \frac{\sum_{j=1}^d\alpha_j\left(3(j+1)^{3/8}+1\right)}{n^{3/8}}.
  \end{equation*}
  
  \medskip
  \noindent\textit{Bounding expected cut-distance:}
  Combining the respective bounds of~\labelcref{eq:partition-bound,eq:simple-bound,eq:first-sampling-bound,eq:random-complex-bound} yields the following:
  \begin{align*}
    \Ex\left[\cutd_{\cut}(\faceted{W},\sample{K}(W[S]);(\alpha_j))\right]
    &\leq \sum_{j=1}^d\alpha_j
      \left(\frac{4\sqrt{d+1}+2^j+2^{j+2}}{\sqrt{\log_2{n}}}
      +\frac{j+2+3(j+1)^{3/8}}{n^{3/8}}
      \right).
  \end{align*}
  Since $\sum_{j=1}^d\alpha_j$ is always a bound for $\cutd(\faceted{W},\sample{K}(W[S]);(\alpha_j))$, it is sufficient to consider $n$ sufficiently large so that the bound in \cref{lemma:sampling} is nontrivial.
  In this regime, we have
  \begin{equation*}
    \Ex\left[\cutd_{\cut}(\faceted{W},\sample{K}(W[S]);(\alpha_j))\right]\leq
    \frac{8\cdot 2^d}{\sqrt{\log_2{n}}}\sum_{j=1}^d\alpha_j.
  \end{equation*}

  Observing that $f(F)=\vert(F)\cdot\cutd(F,\faceted{W};(\alpha_j))/\sum_{j=1}^d\alpha_j$ is a reasonably smooth parameter for any (nonzero) finite, nonnegative sequences $\alpha_1,\ldots,\alpha_d$, \cref{thm:sample-concentration} implies
  \begin{equation*}
    \cutd_{\cut}(\faceted{W},\sample{K}(W[S]);(\alpha_j))\leq
    \frac{8\cdot 2^d+1}{\sqrt{\log_2{n}}}\sum_{j=1}^d\alpha_j
  \end{equation*}
  with probability at least $1-\exp(-n/(2\log_2{n}))$.
  Since $\sample{K}(W[S])$ is identically distributed to $\sample{K}(n,W)$, this yields the desired bound.
\end{proof}

\section{Proof of \Cref{lemma:inverse-counting}}
\label{sec:proofs:inverse-counting}

Given our interest in cut-distances that only consider the first $d$ dimensions, we denote by $\sample{K}_d(n,W)$ the $d$-skeleton of the random simplicial complex $\sample{K}(n,W)$.
Equivalently, if we define $W_d$ as the complexon that is equal to $W$ in all dimensions less than or equal to $d$, and $0$ otherwise, $\sample{K}_d(n,W)$ is equal to $\sample{K}(n,W_d)$ in distribution.

We state a weaker version of \cref{lemma:pre-inverse-counting} in terms of total variation distances first, which we then strengthen to be in terms of homomorphism densities.
\begin{lemma}\label{lemma:pre-inverse-counting}
  Let $U,W$ be two complexons, and suppose that for some $d\geq 1, n\geq 2$, we have
  \begin{equation*}
    \dvar(\sample{K}_d(n,U),\sample{K}_d(n,W))<1-2\exp\left(-\frac{n}{2\log_2{n}}\right),
  \end{equation*}
  where $\dvar$ denotes the total variation distance between probability distributions.
  Then, for all finite nonnegative sequences $\alpha_1,\ldots,\alpha_d$,
  \begin{equation*}
    \cutd_{\cut}(\faceted{U},\faceted{W};(\alpha_j))\leq
    \frac{2^{d+4}+2}{\sqrt{\log_2{n}}}\sum_{j=1}^d\alpha_j.
  \end{equation*}
\end{lemma}

\begin{proof}
  We prove this bound via the probabilistic method.
  By assumption, we can couple $\sample{K}_d(n,U)$ and $\sample{K}_d(n,W)$ so that with probability at least $2\exp(-n/(2\log_2{n}))$,
  \begin{equation}\label{eq:pinvcou-1}
    \sample{K}_d(n,U)=\sample{K}_d(n,W).
  \end{equation}
  By \cref{lemma:sampling}, we have with probability at least $1-\exp(-n/(2\log_2{n}))$,
  \begin{equation}\label{eq:pinvcou-2}
    \cutd_{\cut}(\faceted{U},\sample{K}_d(n,U);(\alpha_j))\leq
    \frac{8\cdot 2^{d}+1}{\sqrt{\log_2{n}}}\sum_{j=1}^d\alpha_j.
  \end{equation}
  Similarly, with probability at least $1-\exp(-n/(2\log_2{n}))$,
  \begin{equation}\label{eq:pinvcou-3}
    \cutd_{\cut}(\faceted{W},\sample{K}_d(n,W);(\alpha_j))\leq
    \frac{8\cdot 2^{d}+1}{\sqrt{\log_2{n}}}\sum_{j=1}^d\alpha_j.
  \end{equation}
  With positive probability, then, \labelcref{eq:pinvcou-1,eq:pinvcou-2,eq:pinvcou-3} hold simultaneously, which implies that
  \begin{align*}
    \cutd_{\cut}(\faceted{U},\faceted{W};(\alpha_j))
    &\leq \cutd_{\cut}(\faceted{U},\sample{K}_d(n,U);(\alpha_j)) \\
    &\quad+ \cutd_{\cut}(\sample{K}_d(n,U),\sample{K}_d(n,W);(\alpha_j)) \\
    &\quad+ \cutd_{\cut}(\sample{K}_d(n,W),\faceted{W};(\alpha_j)) \\
    &\leq \frac{2^{d+4}+2}{\sqrt{\log_2{n}}}\sum_{j=1}^d\alpha_j,
  \end{align*}
  as desired.
\end{proof}

The inverse counting lemma follows, by showing that closeness in homomorphism densities implies the condition of \cref{lemma:pre-inverse-counting}.

\begin{proof}[Proof of \cref{lemma:inverse-counting}]
  For convenience, put $f(n,d)=0.999\cdot 2^{-((1+n)^{d+2})}$.
  Let $F$ be a finite, simplicial complex on $n$ nodes of dimension $d$.
  By the inclusion-exclusion principle, the induced homomorphism density can be written as
  \begin{equation*}
    \tind(F,W_d) = \sum_{G\subseteq\antifaceted{F}:\dim{G}\leq d}\left(-1\right)^{|G|}\thom(F\cup G,W),
  \end{equation*}
  with a similar expression holding for the complexon $U$.
  Since each $F\cup G$ yields a simplicial complex of dimension $d$, we have $|t(F\cup G,U)-t(F\cup G,W)|\leq f(n,d)$ for all such $F,G$, by assumption.
  Observe that for any simplicial complex $F$ on $n$ nodes of dimension $d$, the number of antifacets of $F$ of dimension at most $d$ is bounded by $\binom{n}{d+1}$.
  This implies that
  \begin{equation*}
    \left|\tind(F,U_d)-\tind(F,W_d)\right|\leq
    2^{\binom{n}{d+1}}\cdot f(n,d).
  \end{equation*}
  Then, by \cref{lemma:ind-sample}, for any such $F$,
  \begin{equation*}
    \left|\Pr\left\{\sample{K}_d(n,U)=F\right\}-\Pr\left\{\sample{K}_d(n,W)=F\right\}\right|\leq
    2^{\binom{n}{d+1}}\cdot f(n,d),
  \end{equation*}
  keeping in mind that there exists at least one $F$ such that strict inequality holds.
  Therefore, assuming $n$ is sufficiently large so that the bound on $\cutd_{\cut}(\faceted{U},\faceted{W};(\alpha_j))$ is nontrivial (\ie, less than $\sum_j\alpha_j$), summing over all simplicial complexes $F$ on $[n]$ of dimension $d$ yields the bound
  \begin{align*}
    \dvar(\sample{K}_d(n,U),\sample{K}_d(n,W))
    &= \sum_{F}\left|\Pr\left\{\sample{K}(n,U)=F\right\}-\Pr\left\{\sample{K}(n,W)=F\right\}\right| \\
    &< 2^{\sum_{j=1}^{d+1}\binom{n}{j}}2^{\binom{n}{d+1}}\cdot f(n,d) \\
    &\leq 1-2\exp\left(-\frac{n}{2\log_2{n}}\right).
  \end{align*}
  The lemma follows from \cref{lemma:pre-inverse-counting}.
\end{proof}

\section{Proofs from \Cref{sec:space:compact}}
\label{sec:proofs:compactness}

\begin{proof}[Proof of \cref{lemma:topological-equivalence}]
  Let $\epsilon>0$ and $U\in\kerns$ be given.
  Let $B_\epsilon(U;(\alpha_j))$ be the open ball centered at $U$ of radius $\epsilon$ with respect to the pseudometric $\cutd_{\cut}(\cdot,\cdot;(\alpha_j))$.
  Similarly, let $B_\epsilon(U;(\beta_j))$ be defined in the same way using the sequence $(\beta_j)$ to define the pseudometric.

  By assumption, there exists an $M$ such that
  \begin{align*}
    \sum_{j>M}\alpha_j &< \epsilon/2 \\
    \sum_{j>M}\beta_j &< \epsilon/2.
  \end{align*}
  Put
  \begin{equation*}
    \epsilon' = \frac{\epsilon\cdot\min_{j\leq M}\beta_j}{2M\cdot\max_{j\leq M}\alpha_j},
  \end{equation*}
  noting that $\epsilon'>0$, since both sequences are strictly positive.

  Let $W\in B_{\epsilon'}(U;(\beta_j))$ be given.
  Then, there exists a measure-preserving bijection $\phi:[0,1]\to[0,1]$ such that
  \begin{equation*}
    \labelcutd_{\cut}(U,W^\phi;(\beta_j)_{j\geq 1}) < \epsilon'.
  \end{equation*}
  This implies that for all $j\leq M$,
  \begin{equation*}
    \labelcutd_{\ncut{j}}(U,W^\phi)\leq\frac{\epsilon}{2M\alpha_j}.
  \end{equation*}
  One can then check that this implies $\labelcutd_{\cut}(U,W^\phi;(\alpha_j)_{j\geq 1})<\epsilon$, so that $B_{\epsilon'}(U;(\beta_j))\subseteq B_\epsilon(U;(\alpha_j))$.
  Applying the same argument symmetrically and for arbitrary values of $\epsilon$ yields the fact that the topologies on $\kerns$ induced by both metrics are equal, as desired.
\end{proof}

We use a similar technique to prove \cref{prop:simple-dense}.

\begin{proof}[Proof of \cref{prop:simple-dense}]
  We show that every complexon $W\in\kerns$ arises as the limit of a sequence in $\fkerns$ in the canonical topology.
  By \cref{lemma:topological-equivalence}, we can choose an arbitrary strictly positive, summable sequence $(\alpha_j)_{j\geq 1}$ with which to define the pseudometric $\cutd_{\cut}(\cdot,\cdot;(\alpha_j))$.
  Let $\epsilon>0$ be given, and consider the sequence of random simplicial complexes $(\sample{K}(n,W))_{n\geq 1}$.
  Let $d$ be such that $\sum_{j>d}\alpha_j<\epsilon/2$.
  By \cref{coro:sampling-convergence}, the sequence $\sample{K}(n,W)$ converges to $\faceted{W}$ with respect to the metric $\cutd_{\cut}(\cdot,\cdot;(\alpha_j)_{j\leq d})$ with probability $1$.
  Therefore, there is some $m_0$ such that for all $m>m_0$, it holds (with probability $1$) that $\cutd_{\cut}(\sample{K}(n,W),\faceted{W};(\alpha_j)_{j\leq d})<\epsilon/2$.
  Thus, we also have for all $m>m_0$ that $\cutd_{\cut}(\sample{K}(n,W),\faceted{W};(\alpha_j)_{j\geq 1})<\epsilon$.
  Observe that $W_{\sample{K}(n,W)}=\faceted{W_{\sample{K}(n,W)}}$.
  Since $\epsilon$ was given arbitrarily, this establishes the convergence of $\sample{K}(n,W)$ to $W$ in the canonical topology.
  Of course, by \cref{lemma:finite-to-complexon-cut}, each $\sample{K}(n,W)$ can be equivalently represented by an element of $\fkerns$, establishing the desired result.
\end{proof}

Before proceeding with the proof of \cref{thm:compactness}, we state the following useful result, which is a corollary of \cref{thm:partition}.

\begin{corollary}\label{coro:partition-refinement}
  Let $W$ be a complexon, and let $1\leq m < n, d\geq 1$.
  For every $m$-partition $\partition{Q}$ of $[0,1]$, there is an $n$-partition $\partition{P}$ refining $\partition{Q}$ such that, for all $0\leq j\leq d$,
  \begin{equation*}
    \labelcutd_{\ncut{j}}(W,W_{\partition{P}}) \leq
    \sqrt{\frac{d+1}{\log_2{n/m}}}.
  \end{equation*}
\end{corollary}

\begin{lemma}\label{lemma:cut-compactness}
  The space $\kerns$ with the $\cutd$-topology is compact.
\end{lemma}

\begin{proof}
  By \cref{lemma:topological-equivalence}, the $\cutd$-topology on $\kerns$ is the same for all strictly positive, summable sequences $(\alpha_j)_{j\geq 1}$.
  Thus, we can assume without loss of generality that $\alpha_j=2^{-j}$ for all $j\geq 1$, so that $\sum_{j\geq 1}\alpha_j=1$.

  To prove compactness, we show that every sequence $W_1,W_2,\ldots\in\kerns$ has a convergent subsequence.
  Following the proof of \citep[Theorem~5.1]{Lovasz2007}, for every $n\geq 1$, we can construct via \cref{thm:partition} partitions $\partition{P}_{n,k}$ of $[0,1]$ for $k\geq 1$, with corresponding stepfunctions $W_{n,k}=(W_n)_{\partition{P}_{n,k}}$ that satisfy
  \begin{enumerate}[label=(\roman*)]
  \item $\labelcutd_{\cut}(W_n-W_{n,k};(\alpha_j))\leq 1/k$ \label{item:1}
  \item $\partition{P}_{n,k}$ is refined by $\partition{P}_{n,k+1}$ for all $n,k$ \label{item:2}
  \item $|\partition{P}_{n,k}|=m_k$ depends only on $k$. \label{item:3}
  \end{enumerate}
  To do so, assume that for some $k\geq 0$, $\partition{P}_{n,k}$ is an $m_k$-partition of $[0,1]$ satisfying conditions \ref{item:1} and \ref{item:3}.
  Put
  \begin{align*}
    \epsilon &= \frac{1}{2(k+1)} \\
    M_\epsilon &= \lceil-\log_2\epsilon + 1\rceil \\
    m_{k+1} &= m_k\cdot 2^{4(k+1)^2(M_\epsilon+1)}.
  \end{align*}
  Applying \cref{coro:partition-refinement} with $d=M_\epsilon$, there exists an $m_{k+1}$-partition $\partition{P}_{n,k+1}$ satisfying \ref{item:1} and \ref{item:2}.
  Moreover, $m_{k+1}$ does not depend on the complexon $W_n$, so that $\partition{P}_{n,k+1}$ also satisfies \ref{item:3}.

  For the base case, let $\partition{P}_{n,1}$ be the trivial (indiscrete) partition, which satisfies conditions \ref{item:1} and \ref{item:3} with $m_k=1$.
  By induction, then, there exists a sequence of partitions $\{\partition{P}_{n,k}\}_{k\geq 1}$ satisfying conditions \ref{item:1}--\ref{item:3}.

  Following the approach of \citep[Theorem~5.1]{Lovasz2007}, we can rearrange the points of $[0,1]$ for each $n$ by a measure-preserving bijection so that every partition in every $\partition{P}_{n,k}$ is an interval, while preserving the properties \ref{item:1}--\ref{item:3}.
  Having done this, we replace the sequence $(W_n)_{n\geq 1}$ by a subsequence so that for every $k$, the sequence $W_{n,k}$ converges almost everywhere to a stepfunction $U_k$ with $m_k$ steps as $n\to\infty$.

  To do so, select a subsequence of $W_n$ such that the length of the $i^{\mathrm{th}}$ interval in $P_{n,1}$ converges for all $i\in[m_1]$.
  Then, take a further subsequence such that the value of $W_{n,1}$ converges on the product of the $i_1^{\mathrm{th}}$ and $i_2^{\mathrm{th}}$ interval, for $i_1,i_2\in[m_1]$.
  Repeat this for the products of the intervals indexed by $i_1,\ldots,i_{d+1}\in[m_1]$ for all $d\geq 1$.
  It follows that $W_{n,1}$ converges almost everywhere to a limit stepfunction $U_1$ with $m_1$ steps.

  Repeat this procedure for all $k>1$, taking further and further subsequences so that for all $k$, $W_{n,k}\to U_k$ almost everywhere, where each $U_k$ is a stepfunction on $m_k$ steps.
  This yields the desired subsequence.

  The proof proceeds identically to that of \citep[Theorem~5.1]{Lovasz2007}.
  Let $\partition{P}_k$ denote the partition of $[0,1]$ into the steps of $U_k$.
  By the condition \ref{item:2} of the array $W_{n,k}$, one can show that for every $k<l$, the partition $\partition{P}_{n,l}$ refines $\partition{P}_{n,k}$, and furthermore that $\partition{P}_l$ refines $\partition{P}_k$.
  Moreover, it holds that
  \begin{equation}\label{eq:limiting-refinement}
    U_k = (U_l)_{\partition{P}_k}.
  \end{equation}

  We now seek to establish the almost everywhere convergence of the sequence $U_k$.
  Define a function $g:\djunion_{d\geq 1}[0,1]^{d+1}\to[0,1]$ so that $g([0,1]^{d+1})=\alpha_d$, for all $d\geq 1$.
  Define a probability measure $\mu$ on $\djunion_{d\geq 1}[0,1]^{d+1}$ such that, for any Borel set $\set{A}\subseteq\djunion_{d\geq 1}[0,1]^{d+1}$,
  \begin{equation*}
    \mu(\set{A}) = \int_{\set{A}}g\drv{\lambda},
  \end{equation*}
  where $\lambda$ is the Lebesgue measure.
  Let $X$ be a random point in $\djunion_{d\geq 1}[0,1]^{d+1}$ chosen according to the measure $\mu$.
  Then, the equation \eqref{eq:limiting-refinement} implies that the sequence $U_1(X), U_2(X), \ldots$ is a martingale with bounded terms.
  By the martingale convergence theorem~\citep[see, for instance,][]{Williams1991}, this sequence is convergent with probability $1$.
  That is to say, the sequence of functions $U_1,U_2,\ldots$ converges almost everywhere, with respect to the measure $\mu$.
  Since $\mu$ is absolutely continuous with respect to $\lambda$, this implies convergence almost everywhere with respect to the Lebesgue measure.
  Denote the pointwise limit of this sequence by $U$.
  By the dominated convergence theorem, we have
  \begin{equation*}
    \lim_{k\to\infty}\sum_{j=1}^\infty\alpha_j\int_{[0,1]^{j+1}}|U(x)-U_k(x)|\drv{x} = \lim_{k\to\infty}\int|U-U_k|\drv{\mu} = 0,
  \end{equation*}
  \ie, $U_k\overset{L^1}{\to}U$ with respect to the probability measure $\mu$.

  We now show that the subsequence $W_n$ converges to $U$ in the cut-distance.
  Let $\epsilon>0$ be given arbitrarily.
  Then, there is some $k>3/\epsilon$ such that
  \begin{equation*}
    \sum_{j=1}^\infty\alpha_j\int_{[0,1]^{j+1}}|U(x)-U_k(x)|\drv{x} =
    \int|U-U_k|\drv{\mu} <
    \epsilon/3.
  \end{equation*}
  For this $k$, there is an $n_0$ such that for all $n>n_0$,
  \begin{equation*}
    \sum_{j=1}^\infty\alpha_j\int_{[0,1]^{j+1}}|U_k(x)-W_{n,k}(x)|\drv{x} =
    \int|U_k-W_{n,k}|\drv{\mu} <
    \epsilon/3,
  \end{equation*}
  again by dominated convergence.
  Then, for all $n>n_0$,
  \begin{align*}
    \cutd_{\cut}(U,W_n;(\alpha_j)_{j\geq 1}) &\leq \cutd_{\cut}(U,U_k;(\alpha_j)) \\
                                             &\quad+ \cutd_{\cut}(U_k,W_{n,k};(\alpha_j)) \\
                                             &\quad+ \cutd_{\cut}(W_{n,k},W_n;(\alpha_j)) \\
                                             &\leq \int|U-U_k|\drv{\mu} \\
                                             &\quad+ \int|U_k-W_{n,k}|\drv{\mu} \\
                                             &\quad+ \cutd_{\cut}(W_{n,k},W_n;(\alpha_j)) \\
                                             &\leq \epsilon,
  \end{align*}
  thus proving convergence of the subsequence, as desired.
\end{proof}

We define the map $\faceted{(\cdot)}:\kerns\to\kerns$ that sends any $W\in\kerns$ to $\faceted{W}$, \ie{}, the faceting map.

\begin{lemma}\label{lemma:facet-continuous}
  The faceting map $\faceted{(\cdot)}:\kerns\to\kerns$ is continuous with respect to the $\cutd$-topology on $\kerns$.
\end{lemma}

\begin{proof}
  By \cref{lemma:topological-equivalence}, we can consider an arbitrary strictly positive summable sequence $(\alpha_j)_{j\geq 1}$, so that the cut-distance $\cutd_{\cut}(\cdot,\cdot;(\alpha_j)_{j\geq 1})$ induces the $\cutd$-topology on $\kerns$.
  Let $W_1,W_2,\ldots$ be a convergent sequence $\kerns$ with respect to the cut-metric, with limiting complexon $W$.
  By \cref{lemma:counting}, this implies that for all $n\geq 1$, the sequence of distributions $\sample{K}(n,W_m)$ converges to $\sample{K}(n,W)$ in the total variation distance as $m\to\infty$.

  We now show that $\faceted{W_1},\faceted{W_2},\ldots\to\faceted{W}$ in the cut-metric.
  Let $\epsilon>0$ be given arbitrarily.
  Let $M_\epsilon$ be an integer such that $\sum_{j>M_\epsilon}\alpha_j<\epsilon/2$, and put
  \begin{equation*}
    n = \left\lceil 2^{\left(8\cdot 5^{M_\epsilon+1}/\epsilon\right)^2} \right\rceil.
  \end{equation*}
  By \cref{lemma:sampling}, we have that for any complexon $W$,
  \begin{equation*}
    \cutd_{\cut}(\faceted{W},\sample{K}(n,W);(\alpha_j)_{j\geq 1}) < \frac{\epsilon}{2}
  \end{equation*}
  with probability at least $1-\exp(-n/(2\log_2{n}))$.
  Since the sequence $\sample{K}(n,W_m)\to\sample{K}(n,W)$ converges in the total variation distance as $m\to\infty$, this implies that for some sufficiently large $m_0$, for each $m>m_0$, there is some simplicial complex $F$ such that simultaneously
  \begin{align*}
    \cutd_{\cut}(\faceted{W_m},F;(\alpha_j)_{j\geq 1}) &< \frac{\epsilon}{2} \\
    \cutd_{\cut}(\faceted{W},F;(\alpha_j)_{j\geq 1}) &< \frac{\epsilon}{2}.
  \end{align*}
  By the triangle inequality, this implies $\cutd_{\cut}(\faceted{W_m},\faceted{W};(\alpha_j)_{j\geq 1}) < \epsilon$ for all $m>m_0$.
  Thus, $\faceted{W_1},\faceted{W_2},\ldots\to\faceted{W}$ in the cut-distance, so that the map $\faceted{(\cdot)}$ is continuous, as desired.
\end{proof}

Defining $\faceted{\kerns}$ as the image of $\kerns$ under $\faceted{(\cdot)}$, \cref{lemma:cut-compactness,lemma:facet-continuous} imply that $\faceted{\kerns}$ is a compact subspace of $\kerns$ with respect to the $\cutd$-topology.
The compactness of the space $\kerns$ with the canonical topology follows.

\begin{proof}[Proof of \cref{thm:compactness}]
  From the pseudometric space $(\kerns,\cutd_{\cut})$, where $\cutd_{\cut}$ is defined in terms of an arbitrary strictly positive, summable sequence, take the metric identification $(\overline{\kerns},\overline{\cutd_{\cut}})$, and define the $\cutd$-topology on $\overline{\kerns}$ be the topology induced by the metric $\overline{\cutd_{\cut}}$.
  Since the metric identification preserves compactness $\overline{\kerns}$ is a compact space under the $\cutd$-topology, by \cref{lemma:cut-compactness}.
  Moreover, taking $\overline{\faceted{\kerns}}$ to be the image of $\overline{\kerns}$ under the map $\faceted{(\cdot)}$, we have that $(\overline{\faceted{\kerns}},\overline{\cutd_{\cut}})$ is a compact metric space by \cref{lemma:facet-continuous}.
  In other words, $\faceted{(\cdot)}:\overline{\kerns}\to\overline{\faceted{\kerns}}$ is a quotient map.

  It is clear that the canonical topology on $\kerns$ is equal to the quotient topology by the map $\faceted{(\cdot)}:\overline{\kerns}\to\overline{\faceted{\kerns}}$.
  Since the space $\overline{\kerns}$ with the quotient topology is compact, the canonical topology is compact, as desired.
\end{proof}